\documentclass[a4paper, 12pt]{article}
\usepackage{float}
\usepackage[sort&compress]{natbib}
\bibpunct{(}{)}{;}{a}{}{,} 
\usepackage{rotating}
\usepackage{lscape}
\usepackage{natbib}
\usepackage{bbm}
\usepackage{amsthm, amsmath, amssymb, mathrsfs, multirow, url, bm}
\usepackage{graphicx} 
\usepackage{ifthen} 
\usepackage{amsfonts}
\usepackage[usenames]{color}
\usepackage{fullpage}
\usepackage{caption}
\usepackage{subcaption}

\theoremstyle{plain} 
\newtheorem{thm}{Theorem}

\newtheorem{manualtheoreminner}{Condition}
\newenvironment{manualcondition}[1]{%
  \renewcommand\themanualtheoreminner{#1}%
  \manualtheoreminner
}{\endmanualtheoreminner}

\newtheorem{lem}{Lemma}
\newtheorem{property}{Property}

\theoremstyle{definition}

\newtheorem{asmp}{Assumption}
\newtheorem{model}{Model}

\theoremstyle{remark}

\newcommand{\E}{\mathsf{E}}

\newcommand{\prob}{\mathsf{P}}

\newcommand{\eps}{\varepsilon}

\newcommand{\tr}{\mathrm{tr}}

\newcommand{\iid}{\overset{\text{\tiny iid}}{\sim}}

\newcommand{\nm}{\mathsf{N}}
\newcommand{\N}{\mathsf{N}}

\newcommand{\wish}{\mathsf{W}}
\newcommand{\gam}{\mathsf{Gamma}}

\newcommand{\hatOmega}{\widehat \Omega}
\newcommand{\hatSigma}{\widehat \Sigma}
\newcommand{\tildeSigma}{\widetilde \Sigma}

\newcommand{\Pp}{\mathcal{P}}
\newcommand{\diff}{\mathsf{d}}

 \newcommand{\one}{\mathbbm{1}}
\newcommand{\ve}{\text{vec}}

\title{An empirical $G$-Wishart prior for sparse high-dimensional Gaussian graphical models}
\author{Chang Liu\footnote{Department of Statistics, North Carolina State University; {\tt cliu22@ncsu.edu},  {\tt rgmarti3@ncsu.edu}} \; and \; Ryan Martin$^*$} 
\date{\today}

\begin{document}
 
 \maketitle 

\begin{abstract}
In Gaussian graphical models, the zero entries in the precision matrix determine the dependence structure, so estimating that sparse precision matrix and, thereby, learning this underlying structure, is an important and challenging problem.  We propose an empirical version of the $G$-Wishart prior for sparse precision matrices, where the prior mode is informed by the data in a suitable way.  Paired with a prior on the graph structure, a marginal posterior distribution for the same is obtained that takes the form of a ratio of two $G$-Wishart normalizing constants.  We show that this ratio can be easily and accurately computed using a Laplace approximation, which leads to fast and efficient posterior sampling even in high-dimensions.  Numerical results demonstrate the proposed method's superior performance, in terms of speed and accuracy, across a variety of settings, and theoretical support is provided in the form of a posterior concentration rate theorem.
\smallskip

\end{abstract}

\section{Introduction}
\label{S:intro}

In big-data applications, an important practical question concerns the relationship between variables.  That is, let $X=(X_1,\ldots,X_p)^\top$ be a multivariate quantity of interest, and suppose that the data consists of independent and identically distributed (iid) copies of $X$.  Then the question is: which pairs of variables are associated with each another?  A common way to formulate this question is via a graphical model, where vertices in the graph $G$ correspond to variables and an edge between a pair of vertices indicates an association.  Then answering the above question about which pairs of variables are associated corresponds to learning the underlying graphical structure.  Here, as is often done in applications, we will assume that $X$ is Gaussian, i.e., $X \sim \nm_p(\mu, \Sigma)$, where $\mu$ is the $p$-vector of means and $\Sigma$ is the $p \times p$ positive definite covariance matrix.  Define the precision matrix $\Omega = \Sigma^{-1}$ as the inverse of the covariance matrix.  In this so-called Gaussian graphical model setting, we can say that a pair of variables $(X_i,X_j)$ are conditionally independent given $X \setminus \{X_i, X_j\}$---if and only if the corresponding entry, $\omega_{ij}$, of the precision matrix $\Omega$ is 0 or, equivalently $\Omega \in \Pp_G$, where $\Pp_G$ is the cone of positive definite symmetric matrices having zero entries corresponding to each missing edge in $G$; see \citet{dempster1972covariance}, \citet{lauritzen1989graphical}, \citet{lauritzen1996graphical}, etc.  

When the dimension $p$ is large, reliable inference about the precision matrix and/or the graphical structure is more-or-less hopeless without some restriction on the model complexity.  A common restriction throughout the literature on high-dimensional inference is {\em sparsity} which, in this case, corresponds to an assumption that not too many pairs of variables are associated or, equivalently, that there are not too many edges in the underlying graphical structure or, equivalently, there are lots of zeros in the precision matrix.  Under the sparsity assumption, regularization methods, such as graphical lasso and CLIME, are widely used; see \citet{Buhlmann2006}, \citet{yuan2007}, \citet{huang2007}, \citet{friedman2008}, and \citet{cai2011}. 

At the same time, Bayesian estimation of large covariance/precision matrices also received attention. Corresponding with frequentists' graphical lasso, \citet{khondker2012}  and \citet{wang2012} proposed to put Laplace prior on the off-diagonals and  exponential prior on the diagonals of the precision matrix, leading to a Bayesian version of graphical lasso. \citet{banerjee2015bayesian} developed a shrinkage prior for the precision matrix, i.e., a mixture of point mass at zero and Laplace priors for off-diagonal elements and exponential prior for the diagonals. Their induced posterior concentration rate attained the graphical lasso convergence rate. On the other hand, in contrast to entry-wise priors on elements of the precision matrix,  priors imposed on the entire matrix space are also considered. \citet{dawid1993hyper} introduced hyper inverse Wishart prior for the covariance matrix $\Sigma$. \citet{roverato2000} studied the Cholesky decomposition of hyper inverse Wishart matrices, and introduced the corresponding prior for the precision matrix $\Omega=\Sigma^{-1} \in \Pp_G$. This induced class of prior supported on $\Pp_G$ is known as $G$-Wishart prior. Bayesian estimation methods for large precision matrices based on $G$-Wishart prior are now common in the literature.  For example, \citet{banerjee2014posterior} applied $G$-Wishart prior to estimate (approximately) banded precision matrices, and \citet{xiang2015high} considered general decomposable matrices using $G$-Wishart prior. 


 Despite the $G$-Wishart's desirable conjugacy property, significant challenges remain in the posterior computations.  These stem from the desire to keep the prior relatively diffuse for the sake of ``non-informativeness.''  That is, the precision parameter of the $G$-Wishart prior is taken to be relatively small and, consequently, a Laplace approximation of its normalizing constant is unreliable.  To get around this, we consider an alternative view of ``non-informativeness,'' one that lets the prior center be informed by data, making it possible---and even advantageous---to have a less diffuse prior, whose normalizing constant can be well approximated using Laplace's method, without biasing the posterior distribution.  The use of {\em empirical priors} with informative centers has been explored recently in \citet{martin2014asymptotically}, \citet{martin2017empirical}, \citet{martin2017asymptotically}, \citet{martin2018empirical} and \citet{martin2019empirical} for some specific applications; a general empirical prior framework is presented in \citet{martin2016empirical}.  Motivated by the significant computational benefits that come from being able to justifiably use a Laplace approximation for learning the underlying graphical structure, this paper considers the construction of an {\em empirical $G$-Wishart prior} for use in sparse, high-dimensional, Gaussian graphical models.  At a high level, our proposal is a mixture of $G$-Wishart priors over $G$, where each component prior is a data-dependent $G$-Wishart distribution with mode equal to the maximizer of the likelihood over the cone of $G$-specific precision matrices, or an appropriate sieve therein.  Details are presented in Section~\ref{s:model}.  

From this empirical prior, construction of the corresponding posterior distribution for the precision matrix, via a slight modification of the usual Bayes's formula, is straightforward.  But since our posterior distribution is derived from an unfamiliar approach that uses an empirical prior, it is necessary to investigate the asymptotic properties to confirm that there are no negative effects of our double-use of the data.  Along this line, in Section~\ref{s:convergence}, Theorem~\ref{thm:concentration}, we establish a posterior concentration rate result in Section~\ref{s:convergence}, and our rate matches that achieved elsewhere in the literature \citep[e.g.,][]{rothman2008sparse, banerjee2015bayesian}.  

Since learning the graphical structure is a primary goal of this work, we explore properties of the marginal posterior distribution for the graph in Section~\ref{s:laplace}.  Our motivation for this particular empirical prior construction was to allow the $G$-Wishart precision parameter to be large, so that Laplace's approximation could be justifiably employed.  Towards this, in Theorem~\ref{thm:laplace_error}, we present a bound on the error in Laplace's approximation in this context, improving on the related results presented in \citet{banerjee2015bayesian}.


In Section~\ref{s:simulation}, we carry out some numerical experiments to investigate the performance of our proposed empirical Bayes approach in graph structure learning compared to other competing methods in the literature.  There we demonstrate that our proposed empirical Bayes solution, in addition to being fast to compute, thanks to the accurate Laplace approximation, has superior performance in terms of various metrics across a range of model settings.  In Section~\ref{S:real_data}, we applied our method to a real-world gene expression data, exploring the gene regulatory network rewriting for breast cancer. Our findings coincide with  a number of existing literatures.  Finally, some concluding remarks are made in Section~\ref{S:discuss}, and technical details and proofs are presented in the Appendix.

\section{Model specification}
\label{s:model}

\subsection{Likelihood}

Consider an undirected graph $G=(V, E)$, where $V=\{1,2,...,p\}$ and $E \subset \{(i,j)\in V \times V: i<j\}$ are the vertex and edge sets, respectively. We denote the number of edges $s=|E|$ or for simplicity $|G|$. For a $p$-dimensional vector $X \sim \N_p(0, {\Omega}^{-1})$, we say that ${X}$ follows a Gaussian graphical model with respect to $G$ if the precision matrix $\Omega \in \Pp_G$, where $\Pp_G$ is the cone of positive definite symmetric matrices having zero entries corresponding to each missing edge in $G$, i.e., $\Omega=(\omega_{ij})_{p \times p}$ satisfies $\omega_{ij}=0$ for all $(i,j) \notin E$.  Whenever it is convenient, we will write $\Sigma=\Omega^{-1}$ to denote the covariance matrix corresponding to the inverse of the precision matrix $\Omega$.  

With a slight abuse of notation, write $X$ for the $n \times p$ matrix with rows $X_1^\top,\ldots,X_n^\top$ where $X_i \iid \nm_p(0, \Omega^{-1})$, $i=1,\ldots,n$.  Then the likelihood function is given by
\begin{equation}
  \Omega \mapsto L_n(\Omega)=(2\pi)^{-\frac{np}{2}}|\Omega|^{\frac{n}{2}}\exp\big\{-\tfrac{n}{2}\tr(\hat{\Sigma}\Omega)\big\},
  \label{eq:likelihood}
\end{equation}
where $\hatSigma=n^{-1} X^\top X$ is the sample covariance matrix and $\tr$ is the trace operator.

\subsection{Prior}
\label{SS:prior}

With the assumed graphical structure in the precision matrix, it is natural to reparametrize the full precision matrix $\Omega$ as $(G, \Omega_G)$, explicitly highlighting its structural dependence on the graph $G$ and then the specific form $\Omega_G$ respecting the graphical structure.  This natural reparametrization suggests formulating a prior distribution for $\Omega$ in a hierarchical fashion.  That is, we start with a marginal prior for $G$ and then specify a conditional prior for $\Omega_G$, given $G$.  What will be unique here is that the latter, the conditional prior for $\Omega_G$, will depend on data in a specific way. 

First, we start with the prior for $G$. We consider two different priors for graph structure here. Following \citet{banerjee2015bayesian}, we use a binary indicator to represent the presence/absence of an edge in $G$. More specifically, indicator variable $\gamma(i, j)=1$ if there is an edge between vertices $i$ and $j$; otherwise vertices $i$ and $j$ are disconnected. In a graph with $p$ vertices, the maximum possible number of edge is $\bar{R}_n=\binom{p}{2}$. And these $\bar{R}_n$ binary indicators are considered to be iid Bernoulli with success probability $q$. In order to achieve desired posterior contraction, we also allow truncating the graph size $|G|$, i.e., the number of edges in $G$, to some non-stochastic $\bar{r}_n \le \bar R_n$, where the index $n$ suggests that it may depend on $n$. Then we have,
\begin{equation}
    \pi(G) \propto q^{|G|}(1-q)^{\bar{R}_n-|G|}\mathbbm{1}(|G| \le \bar{r}_n).
    \label{eq:prior_simu}
\end{equation}
More generally, we can explicitly impose a prior on graph size $|G|$,
\begin{align}
    \pi(G) \propto e^{-\tau_n|G|}\mathbbm{1}(|G| \le \bar{R}_n)， \label{eq:prior_G2}
\end{align}
where $\tau_n$ can be viewed as a penalizing parameter on $|G|$ and may depend on $n$. In order to satisfy the conditions for posterior convergence, we choose \[\tau_n=a\log p,\]
where $a$ is a positive constant to be chosen; see Condition~\ref{cond:P}.  If $\tau_n=\log\{(1-q)/q\}$ and $\bar{r}_n=\bar{R}_n$, then the two priors \eqref{eq:prior_simu} and \eqref{eq:prior_G2} are identical.

For a given graph $G$, recall the definition \citet[e.g.,][]{atay2005monte} of the $G$-Wishart distribution: a random matrix $M$, taking values in $\Pp_G$, is said to have a $G$-Wishart distribution, denoted as $M \sim \wish_G(\delta, D)$, depending on a shape parameter $\delta > 2$ and a symmetric, positive definite inverse scale matrix parameter $D$, if the density function can be written as 
\[ \wish_G(M \mid \delta, D) =  I_G(\delta, D)^{-1} \, |M|^{(\delta-2)/2} \exp\{-\tfrac12 \tr(D M)\}, \quad M \in \Pp_G, \]
where $I_G(\delta,D)$ is the normalizing constant, 
\begin{align}
    I_G(\delta,D) = \int_{\Pp_G} |M|^{(\delta-2)/2} \exp\{-\tfrac12 \tr(D M)\} \, \diff M. \label{eq:prior_norm}
\end{align} 
A typical strategy, applied in \citet{banerjee2014posterior}, \citet{xiang2015high}, and elsewhere, is to take the conditional distribution of $\Omega_G$, given $G$, to be $\wish_G(\delta, D)$ where $D$ is some prior guess of $\Omega^{-1}$ and $\delta$ is relatively small, making the prior diffuse or ``non-informative.''  A consequence of choosing $\delta$ small is that computation of the normalizing constant is generally non-trivial; see Section~\ref{s:laplace}.  Here, however, following the references stated in Section~\ref{S:intro}, we will let the data inform the prior center and, consequently, for the sake of computations, we can take $\delta$ relatively large since there is less concern about being ``non-informative.''  In particular, we take our (empirical) conditional prior density for $\Omega_G$, given the graph $G$, as 
\begin{equation}
\label{eq:prior_O}
\pi_n(\Omega \mid G) = \wish_G(\Omega \mid \delta, (\delta-2) \hatOmega_G^{-1}), \quad \Omega \in \Pp_G, 
\end{equation}
where the shape parameter $\delta$ is relatively large (see Section~\ref{s:simulation}) and $\hatOmega_G$ is a (sieve) maximum likelihood estimator, 
\begin{equation}
\label{eq:likelihood}
\hatOmega_G = \arg\max_{\Omega \in \Theta_n(G)} L_n(\Omega), 
\end{equation}
with $\Theta_n(G) \subseteq \Pp_G$ an appropriately chosen sieve. By letting the $G$-Wishart scale matrix be $(\delta-2)\hatOmega_G^{-1}$, it is easy to check that the conditional prior mode is $\hatOmega_G$.  Again, the idea is that the data-driven prior scale matrix would center the prior in a ``good'' spot in $\Pp_G$, so that being diffuse is of lesser importance, hence $\delta$ can be taken relatively large.  In practice, for the sieve, we recommend simply taking $\Theta_n(G) = \Pp_G$ and, in this case, $\hatOmega_G$ is the conditional prior mode, which explains why we can interpret our approach as centering the $G$-specific prior around a reasonable estimator. However, in our theoretical investigations in Section~\ref{s:convergence}, we will require some additional control on the properties of this estimator, which can be achieved with $\Theta_n(G)$ a large but proper subset of $\Pp_G$; see \eqref{eq:sieve} below.  

In the end, when we combine the marginal prior for $G$ and the empirical conditional prior for $\Omega_G$, given $G$, we get a similarly empirical prior for $\Omega = (G, \Omega_G)$.  This prior will be denoted as $\Pi_n$, with the subscript ``$n$'' indicating dependence on data.

\subsection{Posterior}

A reasonable approach would be to use Bayes's formula to combine the (empirical) prior $\Pi_n$ and the likelihood function in \eqref{eq:likelihood} to get a posterior distribution.  However, since the prior also includes data, one might expect that some additional regularization is required to prevent the posterior from tracking the data too closely, i.e., overfitting.  To carry out this extra regularization, one could add an extra penalty in the prior, to discount those parameter values with large likelihood or, alternatively, just take out a small portion of the likelihood.  Following \citet{martin2016empirical}, \citet{martin2017empirical}, and others, we introduce a discount factor $\alpha < 1$ and carry out the Bayesian update with a fractional likelihood, $L_n^\alpha$.  In particular, the conditional posterior for $\Omega_G$, given $G$, is given by 
\[ \pi^n(\Omega \mid G) = \frac{L_n^\alpha(\Omega) \, \pi_n(\Omega \mid G)}{\int_{\Pp_G} L_n^\alpha(\Omega) \, \pi_n(\Omega \mid G) \, \diff \Omega}, \quad \Omega \in \Pp_G. \]
An alternative strategy to carry out this regularization, one that does not directly manipulate the likelihood, is presented in \citet{martin2016empirical}, what they call ``Type~2 regularization,'' but this is more difficult computationally.  Anyway, we recommend using a relatively large $\alpha$, e.g., $\alpha=0.99$, so the numerical results in applications would be no different than those obtained with the more familiar choice of $\alpha=1$.  

Clearly, the $\alpha$ power on the normal likelihood does not affect its normal form, so it is easy to see that the $G$-Wishart prior is conjugate and, therefore, the conditional posterior above is also a $G$-Wishart distribution, i.e., 
\[ \pi^n(\Omega \mid G) = \wish_G(\Omega \mid \delta + \alpha n, \tildeSigma_G), \quad \Omega \in \Pp_G, \]
where $\tildeSigma_G = \alpha n \hatSigma + (\delta-2)\hatOmega_G^{-1}$.  Again, it is not difficult to show that the conditional posterior mode is also $\hat{\Omega}_G$.   

Moreover, for the purpose of learning the graph structure, it is important to get the marginal posterior distribution of $G$.  Given the conjugacy properties of the (empirical) $G$-Wishart prior, it is easy to show that this marginal posterior is given by 
\begin{equation}
\label{eq:post_G}
\pi^n(G) \propto \pi_n(G) \frac{I_G(\delta + \alpha n, \tildeSigma_G)}{I_G(\delta, (\delta-2) \hatOmega_G^{-1})}. 
\end{equation}
In general, neither of the two normalizing constants in the above ratio are available in closed-form, so some kind of numerical approximations are required.  We discuss this computational problem in Section~\ref{s:laplace}.  

In the end, we have a marginal posterior for $G$ and a conditional posterior for $\Omega_G$, given $G$, which can be put together as usual to give a posterior distribution for $\Omega = (G,\Omega_G)$, which we will denote by $\Pi^n$, where the subscript ``$n$'' has been moved {\em up} to a superscript to indicate that the empirical prior has been {\em up}dated to a posterior.  The remainder of the paper is focused on investigating theoretical properties and empirical performance of this posterior distribution $\Pi^n$.

\section{Posterior concentration properties} 
\label{s:convergence}


\subsection{Setup and assumptions}

As is common in the literature on high-dimensional asymptotics, our setup is best described using triangular arrays of random vectors.  That is, for a particular $n$, we have iid samples $X_1,\ldots,X_n$ from $\nm_p(0,\Omega^{\star-1})$, but where the dimension $p=p_n$ and the true precision matrix $\Omega^\star = \Omega_n^\star$ can depend on $n$.  For notational simplicity, we will ignore the dependence of $p$ and $\Omega^\star$ on $n$, but it is important to remember that we are in a genuinely high-dimensional setting where $p \to \infty$ at a to-be-specified rate as $n \to \infty$.  

Next, since we are assuming the true precision matrix to be sparse, this implies existence of a true graph $G^\star$---whose dependence on $n$ is also being suppressed in the notation---that controls the sparsity structure in $\Omega^\star$.  Write $|G^\star|$ for the number of edges in the graph $G^\star$, which implies that the {\em effective dimension} of $\Omega^\star$ is $p + |G^\star|$.  Of course, this effective dimension can be no more than $\binom{p}{2}$, but sparsity suggests that it should be of much smaller order, 
which constrains the complexity of $G^\star$.  

In this section, we are interested in the asymptotic convergence properties of the posterior distribution $\Pi^n$ defined above, as $n \to \infty$.  In particular, we want to show that $\Pi^n$ concentrates asymptotically around the true $\Omega^\star$ in various senses.  Our analysis is based on a few assumptions about the dimension and effective dimension, the graph structure, and the true precision matrix, which we detail below.  

\begin{asmp}
\label{asmp:dimension}
The actual dimension, $p=p_n$, satisfies $p \asymp n^c$ for some $c \in (0,1)$, and the effective dimension satisfies $p+|G^\star|=o(n / \log p)$ as $n \to \infty$.
\end{asmp}

\begin{asmp}
\label{asmp:decomposable}
The true graph $G^\star$ is decomposable. 
\end{asmp}

\begin{asmp}
\label{asmp:eigenvalue}
The smallest and largest eigenvalues of $\Omega^\star$, denoted by $\lambda_{\text{min}}^\star$ and $\lambda_{\text{max}}^\star$, and bounded away from 0 and $\infty$, respectively, i.e.,  
 \[0<\lambda_0 \le \lambda_{\text{min}}^\star \le \lambda_{\text{max}}^\star \le \lambda_0^{-1} < \infty,\]
where $\lambda_0>0$ is a sufficiently small constant; we do not assume $\lambda_0$ is known.
\end{asmp}

Assumption~\ref{asmp:eigenvalue} is not strictly required for the rate result in Theorem~\ref{thm:concentration}, but it aids in the interpretation.  It may be possible to relax the other assumptions, e.g., to cover $p > n$, and we discuss these possible extensions in Section~\ref{SS:remarks} below.  

As we mentioned in Section~\ref{SS:prior}, our theoretical analysis requires some extra control on where the conditional prior for $\Omega$, given $G$, is centered.  We accomplish this by centering on a sieve MLE.  Specifically, we restrict optimization in \eqref{eq:likelihood} to  
\begin{equation}
\label{eq:sieve}
\Theta_n(G) = \{\Omega \in \Pp_G: \xi_n^{-1} \leq \lambda_{\text{min}}(\Omega) \leq \lambda_{\text{max}}(\Omega) \leq \xi_n\},
\end{equation}
where $\lambda_{\text{min}}$ and $\lambda_{\text{max}}$ denote the minimal and maximal eigenvalues of the stated matrix, respectively, and $\xi_n$ is a deterministic sequence such that $\xi_n \to \infty$.  

For the remainder of this section, when we refer to the ``empirical $G$-Wishart prior,'' we mean that where the conditional prior center in \eqref{eq:likelihood} is based on restricting the optimization to the sieve $\Theta_n(G)$ defined in \eqref{eq:sieve}.  We should emphasize, however, that no such constraints appear to be needed for practical implementation, the extra control it provides is only needed for our technical proofs.  Besides, the sieve itself is quite wide, since $\xi_n$ can increase like a polynomial in $p$, so the optimization over $\Theta_n(G)$ would be the same or at least very similar to that over $\Pp_G$ in practice.  

Finally, there are some conditions required on the hyperparameters associated with the priors $\pi(G)$ and $\pi_n(\Omega \mid G)$ in \eqref{eq:prior_G2} and \eqref{eq:prior_O}, respectively.  We summarize these details in {\em Condition P} below, where ``P'' stands for prior.

\begin{manualcondition}{P}
\label{cond:P}
For the conditional prior $\pi_n(\Omega \mid G)$, choose the precision parameter $\delta > 2$ such that $\delta = O(1)$ as $n \to \infty$, and choose the sequence $\xi_n$ in \eqref{eq:sieve} to satisfy $\xi_n \propto p^m$ for some $m \geq 0$.  For the marginal prior $\pi(G)$ in \eqref{eq:prior_G2}, choose $\tau_n = a\log p$ where $a$ is such that $a > 1 + m\delta$.
\end{manualcondition}

\subsection{Recovery}

Similar to those general results in \cite{barron1999consistency}, \cite{ghosal1999posterior}, and \cite{walker2007rates}, our main result establishes the rate of concentration of the posterior $\Pi^n$ for the precision matrix $\Omega$ around the true precision matrix $\Omega^\star$ with respect to Hellinger distance.  That is, we show that the $\Pi^n$-probability assigned to sets of the form $\{\Omega: H(p_{\Omega^\star}, p_\Omega) > M \eps_n\}$ is vanishing, as $n \to \infty$, where $\eps_n$ is called the {\em concentration rate}, $H$ is Hellinger distance, and $p_\Omega$ is the $\nm_p(0,\Omega^{-1})$ density function.  


Although concentration with respect to Hellinger distance is common in the literature, they can be difficult to interpret from the perspective of precision matrix estimation or structure learning.  Therefore, it is worth asking if $\Pi^n$ concentrates in a neighborhood of $\Omega^\star$ relative to a more natural measure of distance between precision matrices.  If so, then, at least intuitively, any reasonable estimator derived from the posterior would also be close to $\Omega^\star$ with respect to that distance.  Results of this type are referred to as {\em recovery} because they pertain directly to estimation of $\Omega^\star$.  

As is common in the literature on precision matrix estimation, consider the Frobenius norm, $\|\cdot\|_F$, for measuring the distance between two matrices, where
\[ \|A\|_F = \{\tr(A^\top A)\}^{1/2}. \]
For example, \cite{rothman2008sparse} show that if the true graph $G^\star$ associated to $\Omega^\star$, then the graphical lasso converges at rate $\{n^{-1} (p+|G^\star|) \log p\}^{1/2}$ in the Frobenius distance. The following theorem shows that the posterior based on our empirical $G$-Wishart prior achieves the same rate in terms of both Hellinger and Frobenius distance.  

\begin{thm}
\label{thm:concentration}
Suppose that the inputs to the empirical $G$-Wishart prior satisfy Condition~\ref{cond:P}.  If Assumptions~\ref{asmp:dimension}--\ref{asmp:decomposable} hold, then there exists a constant $M > 0$ such that 
\[ \sup_{\Omega^\star \in \Pp_{G^\star}} \E_{\Omega^\star} \Pi^n\{\Omega: H(p_{\Omega^\star}, p_{\Omega}) > M \eps_n(G^\star)\} \to 0, \quad n \to \infty, \]
where $\eps_n^2(G^\star) = n^{-1} (p + |G^\star|) \log p$.  If, in addition,  Assumption~\ref{asmp:eigenvalue} holds, then Hellinger distance can be replaced by Frobenius distance, i.e., 
\[ \sup_{\Omega^\star \in \Pp_{G^\star}} \E_{\Omega^\star} \Pi^n\{\Omega: \|\Omega - \Omega^\star\|_F > M \eps_n(G^\star)\} \to 0, \quad n \to \infty. \]
\end{thm}

\begin{proof}
See Appendix~\ref{SS:proof.concentration}.
\end{proof}



\subsection{Effective dimension}

Of course, since $\Omega$ is $p \times p$, the literal dimension of our parameter is $\binom{p}{2}$.  However, the true precision matrix is believed to be sparse so, ideally, the {\em effective dimension} of the parameter---as measured by the posterior distribution of $|G_\Omega|$, where $|G_\Omega|$ is the graph associated with a precision matrix $\Omega$---would be significantly smaller.  More specifically, it would be desirable if the posterior concentrated on $\Omega$ such that $|G_\Omega|$ were roughly equal to $|G^\star|$.  The following theorem establishes that the posterior based on our empirical $G$-Wishart prior roughly achieves this.  

\begin{thm}
\label{thm:dimension}
Suppose the empirical prior satisfies Condition~\ref{cond:P} and, for the hyperparameters $(a,m,\delta)$ specified there, define 
\[ \rho(t) = \frac{a+(m+t^{-1})\delta+2t^{-1}}{a-m\delta-1}, \quad t > 0. \]
Then under Assumptions~\ref{asmp:dimension}--\ref{asmp:decomposable}, for any $\rho > \rho(c)$, where $c$ is as in Assumption~\ref{asmp:dimension}, the posterior based on the empirical $G$-Wishart prior satisfies 
\[ \sup_{\Omega^\star \in \Pp_{G^\star}} \E_{\Omega^\star}\Pi^n\{\Omega: |G_\Omega| \geq \rho \max(p, |G^\star|) \} \to 0, \quad n \to \infty. \] 
\end{thm}

\begin{proof}
See Appendix~\ref{SS:proof.dimension}.
\end{proof}

Theorem~\ref{thm:dimension} basically states that, by using the prior $\pi(G) \propto e^{-\tau_n |G|}$ with $\tau_n$ as in Condition~\ref{cond:P}, the corresponding marginal posterior for $G$ concentrates on graphs whose dimension is close to $|G^\star|$.  This has at least two interesting consequences.  First, observe that an oracle with knowledge of $|G^\star|$ would use a prior for $G$ supported on graphs with roughly that known complexity, and then a result similar to that in Theorem~\ref{thm:dimension} would hold automatically by construction.  \citet{banerjee2015bayesian} basically take this oracle approach because they truncate their prior at a multiple of $n\eps_n^2/\log n$, where $\eps_n = \eps_n(G^\star)$ depends on the unknown complexity of $G^\star$.  Of course, none of us are oracles, so this approach is not really practical; but Theorem~\ref{thm:dimension} says we do not need to be oracles to learn the effective dimension of $\Omega^\star$, at least approximately.  Second, combining this with the recovery result in the previous subsection, we have every reason to believe that the posterior will be successful at learning the graph structure in $G^\star$.  That is, by Theorem~\ref{thm:concentration}, we know that the posterior is close to $\Omega^\star$ in a certain sense but, on its own, does not imply that the posterior for $G$ is concentrating on $G^\star$.  It does basically suggest that the posterior for $G$ is not underfitting, i.e., missing edges present in $G^\star$.  Then Theorem~\ref{thm:dimension} says that the posterior is not drastically overfitting, so the most likely explanation is that the posterior for $G$ is concentrating on $G^\star$.  A formal proof of a ``graph selection consistency theorem'' presently escapes us, but this is apparently quite challenging since, to our knowledge, there are currently no such results in the Bayesian literature.

\subsection{Technical remarks}
\label{SS:remarks}

Here we discuss the various ways that our results could potentially be strengthened by weakening the assumptions.  First we consider Assumption~\ref{asmp:decomposable} about the true graph $G^\star$ being decomposable.  This assumption is already relatively mild, in particular because our posterior distribution is defined on all graphs, decomposable or not, but it is still worth discussing the role this assumption plays and how it could be weakened or removed altogether.  The only point in the proof this assumption is needed is in lower bounding the normalizing constant of the posterior $\Pi^n$ to prove Lemma~\ref{lem:D_n}.  There we required a bound on moments of a $G$-Wishart random matrix with $G=G^\star$, and we made use of Lemma~3.3 from \citet{xiang2015high} which assumes the underlying graph is decomposable.  If this moment could be bounded through some other means, without assuming $G^\star$ is decomposable, then Assumption~\ref{asmp:decomposable} could be dropped from our theorems above.  

Second, is it possible to remove Assumption~\ref{asmp:dimension} and cover the extreme high-dimensional case with $p > n$?  In contrast to the discussion in the previous paragraph, here the main obstacle is practical rather than theoretical.  Our empirical prior construction requires centering each $G$-specific prior on a (sieve) maximum likelihood estimator but, when $p > n$, existence of such estimators is not guaranteed \citep[e.g.,][Corollary 2.3]{uhler2012geometry}.  Of course, we could avoid this issue by restricting the support of our empirical prior to those $G$ for which the corresponding maximum likelihood estimator exists, assume that $G^\star$ is in this support asymptotically, and prove a recovery result like in Theorem~\ref{thm:concentration}.  We have opted not to do this here because the use of such a complicated prior distribution makes the already non-trivial computation of the posterior even more difficult, thereby limiting the practicality of our proposal.  



\section{Computation}
\label{s:laplace}

\subsection{Laplace approximation}

Marginal posterior probability for the graph $G$ \eqref{eq:post_G} provides a direct way for graphical structure learning, that is, we can compute $\pi^n(G)$ for as large a collection of graphs $G$ as possible to identify those with the highest posterior probability.  However,  \eqref{eq:post_G} involves posterior and prior normalizing constants in the numerator and denominator, respectively, which are high-dimensional integrals over the positive semi-definite cone $\Pp_G$. When $G$ is a decomposable graph, we can still compute them through decomposition of subgraphs by finding its maximal cliques \citep{lenkoski2011computational}; on the contrary, when $G$ is non-decomposable, which is often the case in MCMC samples, it is very difficult to find a closed form. \cite{atay2005monte} propose a Monte Carlo method for approximating the normalizing constant. Given that these integrals involve a unimodal and ``Gaussian-like'' integrand, the use of a Laplace approximation is also attractive.  

Most papers on $G$-Wishart priors recommend the use of some default scale matrix, e.g., identity, and to make this arbitrary choice have relatively mild influence on the posterior, they take the shape parameter $\delta$ to be  small, say $\delta=3$ or 4, making the prior relatively diffuse.  However, in our empirical $G$-Wishart setting, we allow the prior mode to be located at the MLE which is a reasonable guess of the true precision matrix. Therefore, larger $\delta$ which helps the prior to concentrate more on the MLE should also be acceptable and maybe even preferred. When $\delta$ is moderately large, the error of Laplace approximation for the prior normalizing constant should hopefully be controlled.  
Given that Laplace method is much simpler and significantly more computationally efficient than Monte Carlo, we propose to apply Laplace approximation to both $I_{G}(\delta, (\delta-2) \hatOmega_{G}^{-1})$ and $I_{G}(\delta+\alpha n, \tildeSigma_{G})$ at $\hatOmega_G$, which is the prior and posterior mode. 

{
Define the function 
\begin{equation}
    h(\Omega)=\log|\Omega|-\tr(\hat{\Omega}_G^{-1} \Omega),
    \label{eq:h_fun}
\end{equation}
so that $G$-Wishart prior normalizing constant
\begin{equation}
    I_G(\delta, (\delta-2)\hat{\Omega}^{-1}_G) = \int_{\Pp_G} e^{(\delta-2)h(\Omega)/2} \, \diff \Omega,
    \label{eq:I_prior}
\end{equation}  
and it is not difficult to show (see Appendix~\ref{SS:laplace_error}) that the corresponding posterior normalizing constant is 
\begin{equation}
    I_G(\delta+\alpha n, \tilde{\Sigma}_{G})=\int_{\Pp_G} e^{(\delta+\alpha n-2)h(\Omega)/2} \, \diff \Omega.
    \label{eq:I_posterior}
\end{equation}
Note that $h(\Omega)$ is a strictly concave function over $\Pp_G$ with global maximum at $\hatOmega_G$.  Negative Hessian of $h(\Omega)$ evaluated at $\hatOmega_G$ can be written as,
\begin{equation}
    Q(\hatOmega_G)=\{ \tr\big(\hatOmega_G^{-1}E_{(i,j)}\hatOmega_G^{-1}E_{(l,m)}\big)\}_{(p+|G|)\times (p+|G|)}. \notag 
\end{equation}
where $E_{(i,j)}=\big(\one_{\{(i,j), (j,i)\}}(l,m)\big)_{p \times p}$ is a symmetric indexing matrix for $(i,j)$, $(l,m) \in E\cup \{(i, j) \in V\times V: i=j\}$.

Then using Laplace approximation based on Taylor series expansion of function $h(\Omega)$, we can approximate the prior and posterior normalizing constants by
\begin{equation}
    I_G(\delta, (\delta-2)\hatOmega^{-1}_G) \approx e^{(\delta-2)h(\hatOmega_G)/2}|Q(\hatOmega_G)|^{-\frac{1}{2}}\big(\tfrac{4\pi}{\delta-2}\big)^{\frac{p+|G|}{2}}, \label{eq:I_laplace}
\end{equation}
and
\[I_G(\delta+\alpha n, \tildeSigma_{G}) \approx
e^{(\delta+\alpha n-2)h(\hatOmega_G)/2}|Q(\hatOmega_G)|^{-\frac{1}{2}}\big(\tfrac{4\pi}{\delta+\alpha n-2}\big)^{\frac{p+|G|}{2}}.\]
Hence, the marginal posterior probability for graph $G$ in \eqref{eq:post_G} can be approximated as,
\begin{equation}
    \pi^n(G) \propto \pi_n(G)L_n^{\alpha}(\hatOmega_G)\big(\tfrac{\delta-2}{\delta+\alpha n -2}\big)^{\frac{p+|G|}{2}}, 
    \label{eq:post_G_approximate}
\end{equation}
where $L_n$ is the likelihood function defined in \eqref{eq:likelihood}. Note that by approximating prior and posterior normalizing constants simultaneously at their common mode $\hatOmega_G$, when taking ratio of the two quantities,  Hessian matrix $Q(\hatOmega_G)$ which is most computationally expensive gets canceled nicely. In practice, this significantly eases our posterior computation compared to other Bayesian methods in which Laplace method is also employed, e.g., \citet{dobra2011copula} and  \citet{banerjee2015bayesian}.

\subsection{Numerical experiments}

Here we show a few comparisons to illustrate the accuracy and numerical efficiency of  Laplace approximation for $G$-Wishart normalizing constant. Since the posterior normalizing constant resembles prior normalizing constant (see \eqref{eq:I_prior} and \eqref{eq:I_posterior}) and can be approximated in a similar way, here we only present the numerical results for prior normalizing constant approximation using \eqref{eq:I_laplace}.

First, we consider a decomposable graph case, namely, the graph corresponding to the precision matrix describing a second-order autoregressive correlation structure; see Model~\ref[model:model2]{2} below.  In this case, since the precision matrix is banded and the graph is decomposable, an analytical expression for the $G$-Wishart normalizing constant is available; see \citet{banerjee2014posterior}. For $p=30$, 50, and 100, and a range of values for the shape parameter $\delta$, we calculate the true prior normalizing constant $I=I_{\text{\sc true}}$ using their analytical formula and the Laplace approximation $\hat I = \hat{I}_{\text{\sc laplace}}$ using Theorem~\ref{thm:laplace_error}. We also calculate the relative error of logarithmic normalizing constant, i.e.,
\[ \mathsf{re}(I,\hat I) =  |\log I - \log \hat I|/\log I.\]
The results are summarized Figures~\ref{fig:ar2_logI} and~\ref{fig:ar2_error}.  The key observation is that the Laplace approximation is very accurate across the entire range of $\delta$, even when the dimension is high.  As expected, the approximation is better for larger $\delta$ and, in fact, the relative error is monotonically decreasing in $\delta$.  But given that our informative, data-driven centering alleviates the need to take $\delta$ small for the sake of non-informativeness, we can safely take $\delta$ in the range of, say, 10 where the Laplace approximation is very accurate.  

\begin{figure}[t]
\centering
    \begin{subfigure}[b]{0.33\textwidth}
                \centering
                \includegraphics[width=0.99\textwidth]{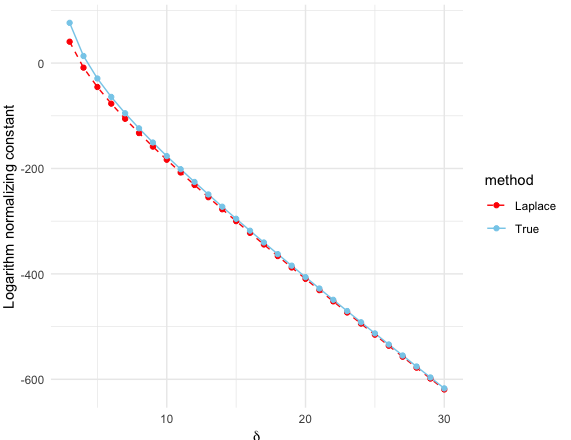}
                \caption{$p=30$}
    \end{subfigure}%
    \begin{subfigure}[b]{0.33\textwidth}
                    \centering
                \includegraphics[width=0.99\textwidth]{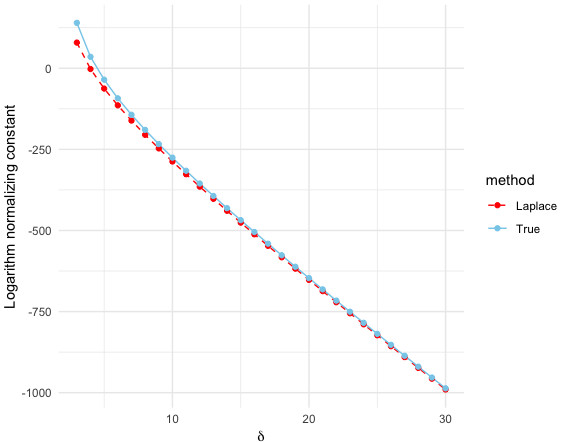}
                \caption{$p=50$}
    \end{subfigure}%
    \begin{subfigure}[b]{0.33\textwidth}
                    \centering
                \includegraphics[width=0.99\textwidth]{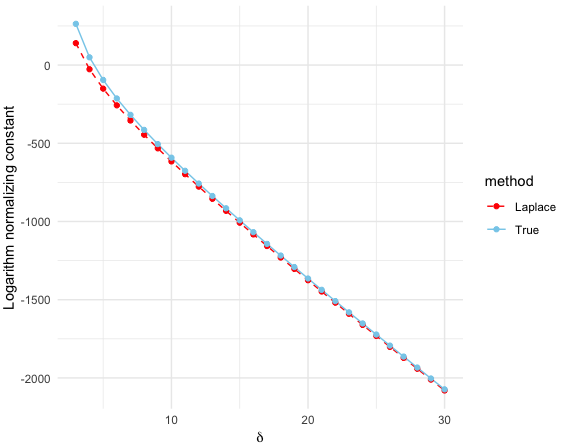}
                \caption{$p=100$}
    \end{subfigure}%
    \caption{Logarithm of prior normalizing constant: $I_{\text{\sc true}}$ versus $\hat{I}_{\text{\sc laplace}}$.}
\label{fig:ar2_logI}
\end{figure}

\begin{figure}[t]
\centering
    \begin{subfigure}[b]{0.33\textwidth}
                \centering
                \includegraphics[width=0.85\textwidth]{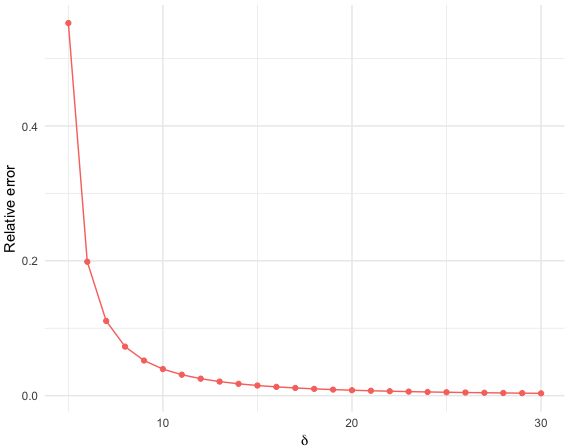}
                \caption{$p=30$}
    \end{subfigure}%
    \begin{subfigure}[b]{0.33\textwidth}
                    \centering
                \includegraphics[width=0.85\textwidth]{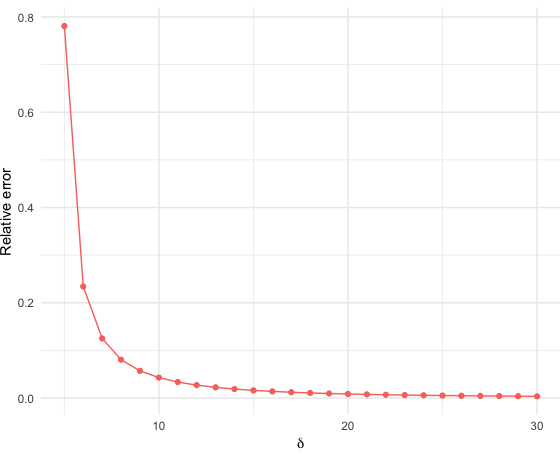}
                \caption{$p=50$}
    \end{subfigure}%
    \begin{subfigure}[b]{0.33\textwidth}
                    \centering
                \includegraphics[width=0.85\textwidth]{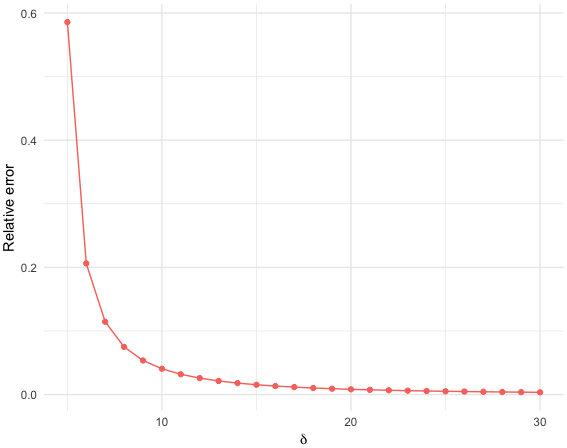}
                \caption{$p=100$}
    \end{subfigure}%
    \caption{\color{black} $\mathsf{re}(I,\hat I)$: True versus Laplace approximation.}
\label{fig:ar2_error}
\end{figure}

Next, we look at the case with nondecomposable graphs. We generate a precision matrix with random sparsity structure using the method in \citet{cai2011}; see Model~\ref[model:model4]{4} below.  For such cases, no analytical formula is available for the $G$-Wishart normalizing constants, so Monte Carlo methods, such as that in \citet{atay2005monte}, are needed.  In this case, we use Monte Carlo to obtain $\tilde I = \tilde I_{\text{\sc mc}}$ and, again, compare to the Laplace approximation $\hat I = \hat I_{\text{\sc laplace}}$.  Log-normalizing constants and relative errors are showed in Figures~\ref{fig:random_logI} and~\ref{fig:random_error}, respectively.  Here, again, we see that the Laplace approximation is very accurate compared to Monte Carlo over the entire range of $\delta$ and for each dimension $p$.  The relative error plot is not monotone decreasing in this case because the Monte Carlo approximations have their own inherent variability, but the errors are still very small, especially for $\delta$ in the range of, say, 10.  In addition, we are interested in computational efficiency of Laplace approximation compared to Monte Carlo.  For Laplace method and Monte Carlo, we plot the average computation time in log-nanoseconds ($1\text{s} = 10^9\text{ns}$) over 500 replications across different dimensions in Figure~\ref{fig:time}.  As expected, the Laplace approximation is substantially faster to compute than Monte Carlo, especially at large $p$.  And apparently this gain in computational efficiency requires no sacrifice in accuracy.  

\begin{figure}[t]
\centering
    \begin{subfigure}[b]{0.33\textwidth}
                \centering
                \includegraphics[width=0.99\textwidth]{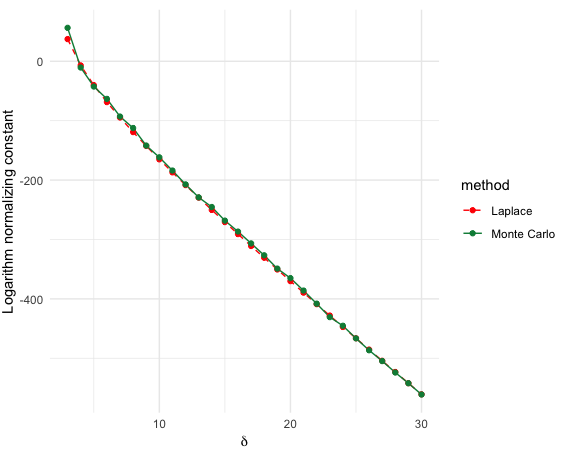}
                \caption{$p=30$}
    \end{subfigure}%
    \begin{subfigure}[b]{0.33\textwidth}
                    \centering
                \includegraphics[width=0.99\textwidth]{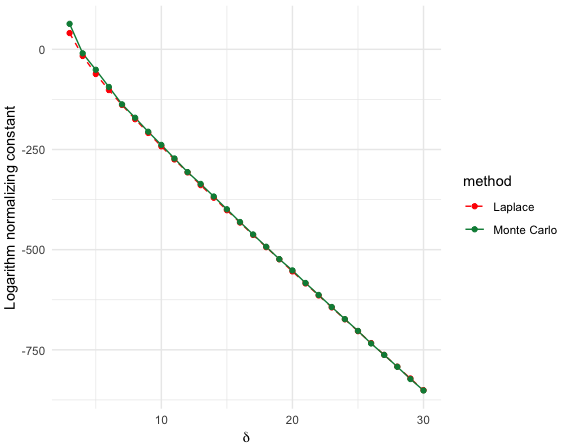}
                \caption{$p=50$}
    \end{subfigure}%
    \begin{subfigure}[b]{0.33\textwidth}
                    \centering
                \includegraphics[width=0.99\textwidth]{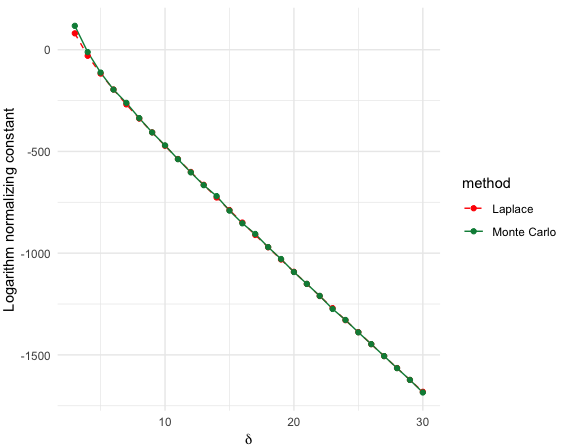}
                \caption{$p=100$}
    \end{subfigure}%
    \caption{Logarithm of prior normalizing constant: $\tilde{I}_{\text{\sc mc}}$ versus $\hat{I}_{\text{\sc laplace}}$.}
\label{fig:random_logI}
\end{figure}

\begin{figure}[t]
\centering
    \begin{subfigure}[b]{0.33\textwidth}
                \centering
                \includegraphics[width=0.85\textwidth]{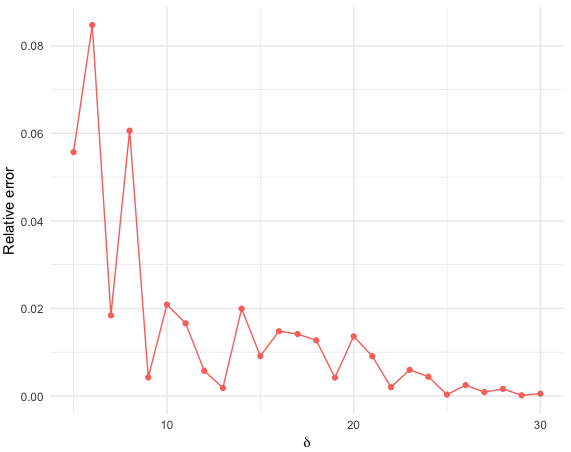}
                \caption{$p=30$}
    \end{subfigure}%
    \begin{subfigure}[b]{0.33\textwidth}
                    \centering
                \includegraphics[width=0.85\textwidth]{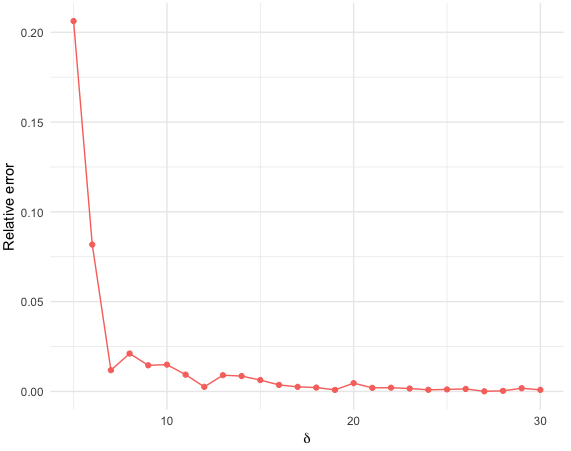}
                \caption{$p=50$}
    \end{subfigure}%
    \begin{subfigure}[b]{0.33\textwidth}
                    \centering
                \includegraphics[width=0.85\textwidth]{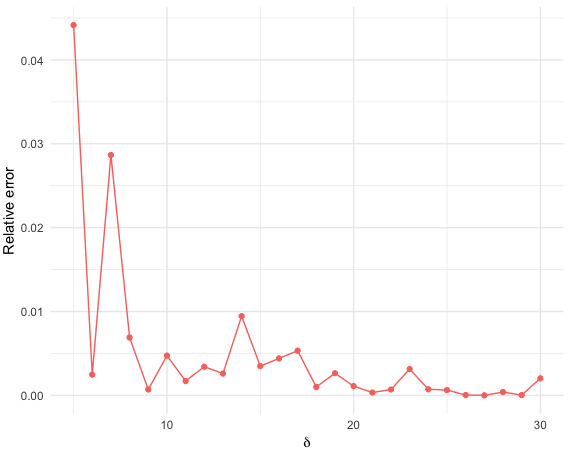}
                \caption{$p=100$}
    \end{subfigure}%
    \caption{\color{black} $\mathsf{re}(I,\hat I)$: Monte Carlo method versus Laplace approximation}
\label{fig:random_error}
\end{figure}

\begin{figure}[t]
\centering
    \includegraphics[width=0.50\textwidth]{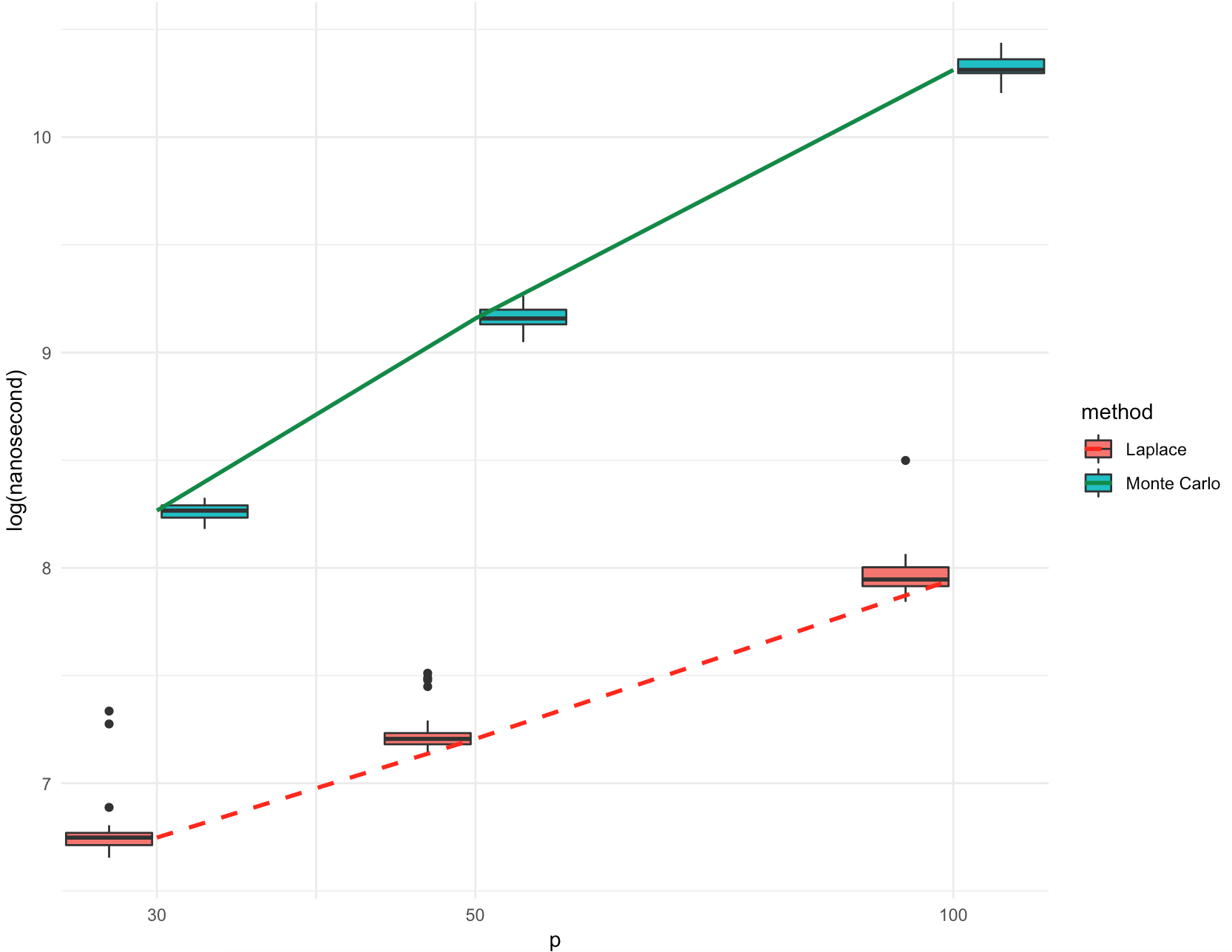}
                \caption{Computation time in log(nanosecond): Laplace approximation versus Monte Carlo}
\label{fig:time}
\end{figure}

\subsection{Error in Laplace approximation}

As \citet{shun1995laplace} point out that, Laplace approximation is not a valid asymptotic approximation when the dimension of the integral is comparable with $n$, and they suggest that it is still reliable when dimension of the integral is $o(n^{1/3})$. In our problem, the dimension of the integral is the number of free parameters, which is $p+|G|$. The following theorem gives a condition when the error in Laplace approximation for posterior normalizing constant is asymptotically negligible, which also matches the conclusion of \citet{shun1995laplace}.


\begin{thm}
\label{thm:laplace_error}

Under Assumption~\ref{asmp:dimension}, for any graph G,  
if $n^{-1/2}\xi_n^8(p+|G|)^{3/2}\log^{3/2} p \to 0$, the error 
in the Laplace approximation for the posterior normalizing constant, $I_G(\delta+\alpha n, \tildeSigma_{G})$, goes to 0  in $\prob_{\Omega^\star}$-probability as $n \to \infty$. 
\end{thm}
\begin{proof}
See Appendix~\ref{SS:laplace_error}.
\end{proof}


Note that $\xi_n$ is the eigenvalue bound in sieve \eqref{eq:sieve} and in Condition~\ref{cond:P}, we require that $\xi_n \propto p^m$ with $m \ge 0$. Hence, with a proper choice of $m$, our Laplace method can achieve asymptotic approximation accuracy for posterior normalizing constant, if $p+|G|= o(n^{1/3})$. This improves the results in Theorem 4.4 from \citet{banerjee2015bayesian}, which actually requires $p+|G|=o(n^{1/5})$. 

In addition, interestingly, note that the remainder in Taylor series expansions takes the same form for both prior and posterior normalizing constants, and the only difference between them is some positive coefficients. Therefore, the sign of the remainder and, hence, the direction (positive or negative) of Laplace approximation error for prior and posterior normalizing constants should be the same. Thus, we can expect that some error can be further canceled when taking ratio of the two normalizing constants when approximating the marginal posterior probability for $G$.

\section{Simulated-data comparisons}
\label{s:simulation}


\subsection{Methods and implementation}

For our simulation study, we consider three alternatives to our proposed method based on the empirical $G$-Wishart prior.  The first is the graphical lasso in \citet{friedman2008}, and we make use of the R packages {\tt glasso} and {\tt qgraph} for our implementation, with tuning parameter chosen via extended BIC.  Second, we consider the Bayesian graphical lasso in \citet{banerjee2015bayesian}, where a Laplace approximation is used to evaluate the posterior distribution for $G$, and a Metropolis--Hastings algorithm is employed to obtain MCMC samples.  For simplicity, a symmetric proposal distribution $q(G' \mid G)$ is used here, which samples $G'=(V, E')$ uniformly from the graphs that differ from $G=(V, E)$ in one position. More specifically, if we use binary indicators $\gamma(i, j)$ and $\gamma'(i, j)$ to represent presence/absence of the edge between vertices $i$ and $j$ in $G$ and $G'$ respectively, in each iteration, there is an equal  probability to either set $\gamma'(i, j)=1$ for one pair of $(i, j)$ uniformly sampled from $V\times V\setminus E$, i.e., $\gamma(i,j)=0$,  or set $\gamma'(i, j)=0$  for one pair of $(i, j)$ uniformly sampled from $E$, i.e., $\gamma(i,j)=1$. And for the other pairs $(-i, -j)$, we simply let $\gamma'(-i,-j)=\gamma(-i,-j)$.  The tuning parameter $\rho$ is set to $0.4$ as recommended by those authors, and we use $q=0.45$ for the graph prior in \eqref{eq:prior_simu}.  Then, $2 \times 10^4$ runs of MCMC with additional $4000$ runs of burn-in are computed to generate the sample of $G$. Finally, a median probability model is selected by setting $\gamma(i, j)=1$, if marginal posterior probability of $\gamma(i, j)=1$ is greater than $0.5$. Third, for a $G$-Wishart prior using identity scale matrix, a reversible jump algorithm in \citet{dobra2011copula} is employed to obtain MCMC samples of $G$. The shape parameter $\delta$ is set to $4$.  The Bernoulli graph prior in \eqref{eq:prior_simu} with $q=0.45$ is also used here. With $2 \times 10^4$ runs of MCMC with additional $4000$ runs of burn-in, we select the median probability model in the same way as in Bayesian graphical lasso.  Finally, for the empirical $G$-Wishart method,  with $\alpha=0.99$, we consider $\delta=4, 10$ and set $q=0.45$ in \eqref{eq:prior_simu}.  To sample the posterior of $G$ in \eqref{eq:post_G}, we applied the same Metropolis--Hastings algorithm described above, with length of $2\times 10^4$ and $4000$ burn-ins.  In the end, we also select the median probability model as our estimated graphical structure.

\subsection{Settings}

For generating data, we consider the following four different models for the true graph and precision matrix.  Examples of the precision matrix and corresponding graph for each of the four settings, with $p=30$, are depicted in Figures~\ref{fig:ar1}--\ref{fig:random1}.  For each of Models~1--4, a sample $X_1,\ldots,X_n$ is generated from $N_p({0},\Omega^{\star -1})$, with the sample size $n=100$ and dimension $p=30, 50, 100$. 

\begin{model}
\label{model:model1}
AR(1) model, where the entries in the covariance matrix $\Sigma^\star$ are given by $\sigma_{ij}=0.7^{|i-j|}$; then the diagonal entries of $\Omega^\star=\Sigma^{\star -1}$ are standardized.
\end{model}

\begin{model}
\label{model:model2}
AR(2) model, with $\omega_{ii}=1$,   $\omega_{i,i-1}=\omega_{i-1,i}=0.5$, $\omega_{i,i-2}=\omega_{i-2,i}=0.25$.
\end{model}

\begin{model}
\label{model:model3}
Star model, with $\omega_{ii}=1$, $\omega_{1,i}=\omega_{i,1}=0.1$ for $i>1$ and $\omega_{ij}=0$. 
\end{model}

\begin{model}
\label{model:model4}
Following \citet{cai2011}, set $\Omega^\star=\frac12(B + B^\top) + \tau I$, where $B$ is a $p \times p$ matrix, in which all diagonals equal $0$ and each off-diagonal entry is independently assigned to be $0.5$ with probability $0.05$ and 1 with probability $0.95$. Then $\tau$ is chosen such that the condition number of $\Omega^\star$ equals $p$, and $\Omega^\star$ is standardized to have unit diagonals. 
\end{model}

\begin{figure}[p]
\centering
    \begin{subfigure}[b]{0.5\textwidth}
                \centering
                \includegraphics[width=0.60\textwidth]{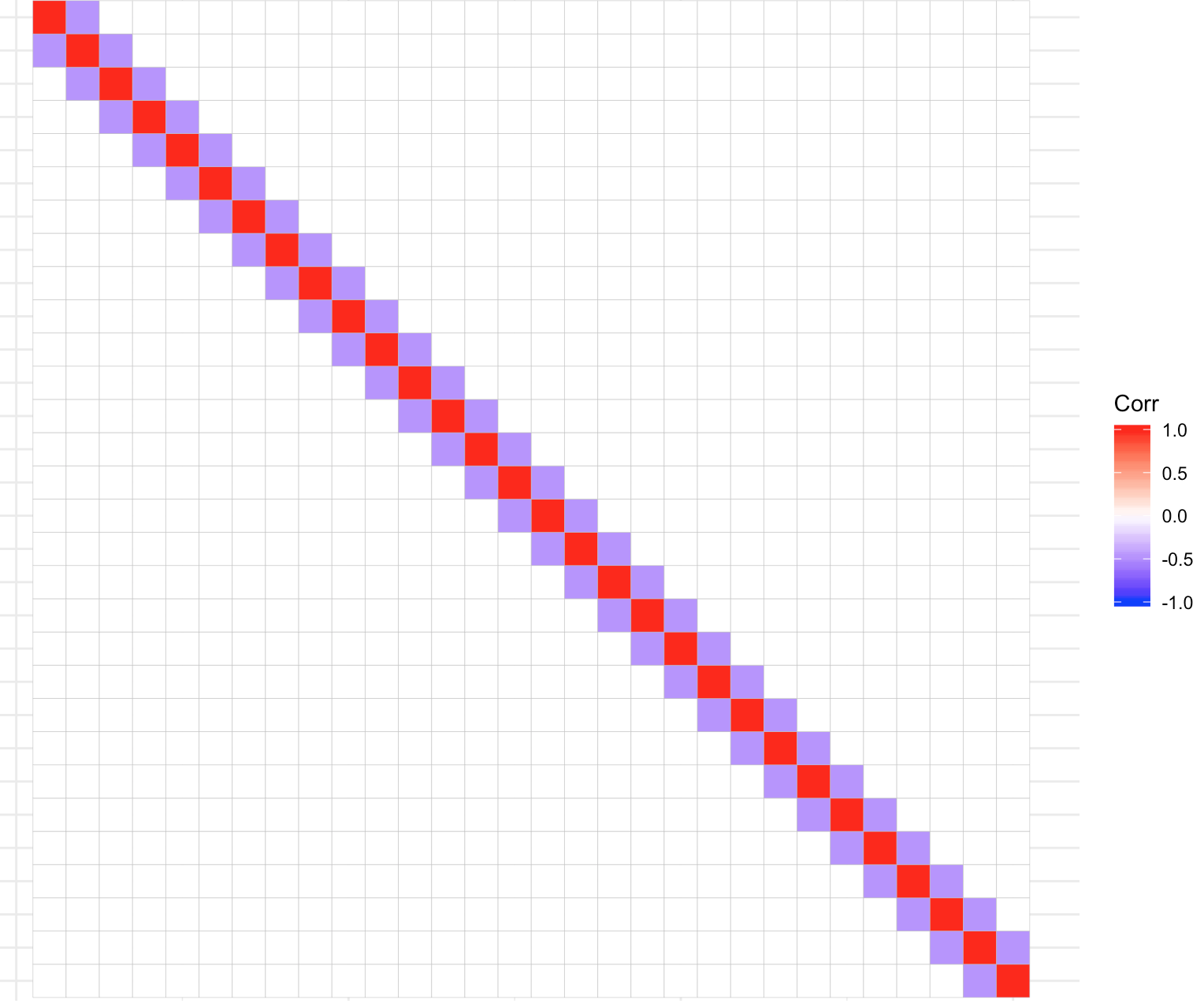}
                \caption{Precision matrix: $\Omega^\star$}
    \end{subfigure}%
    \begin{subfigure}[b]{0.5\textwidth}
                    \centering
                \includegraphics[width=0.60\textwidth]{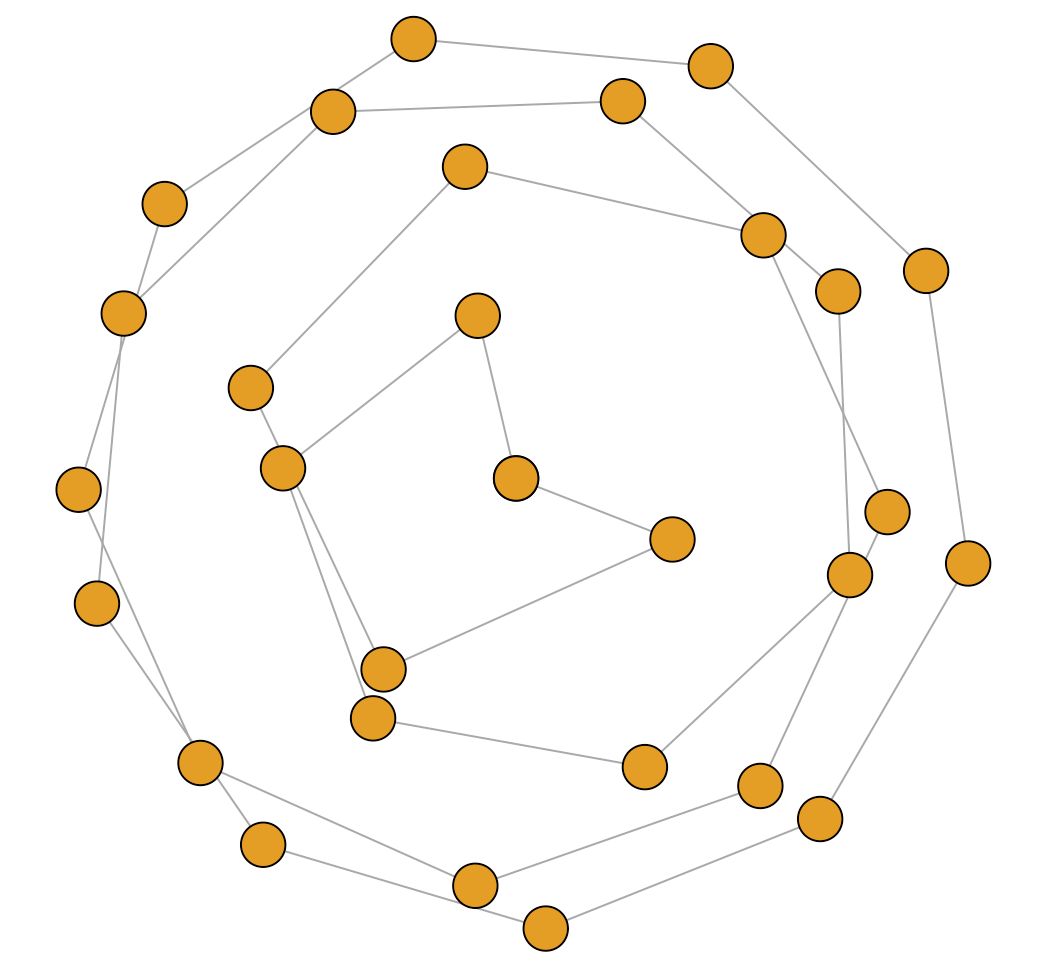}
                \caption{Graph: $G^\star$}
    \end{subfigure}%
    \caption{Precision matrix of Model 1 and its induced graph, $p=30$}
\label{fig:ar1}
\end{figure}

\begin{figure}[p]
\centering
    \begin{subfigure}[b]{0.5\textwidth}
                \centering
                \includegraphics[width=0.60\textwidth]{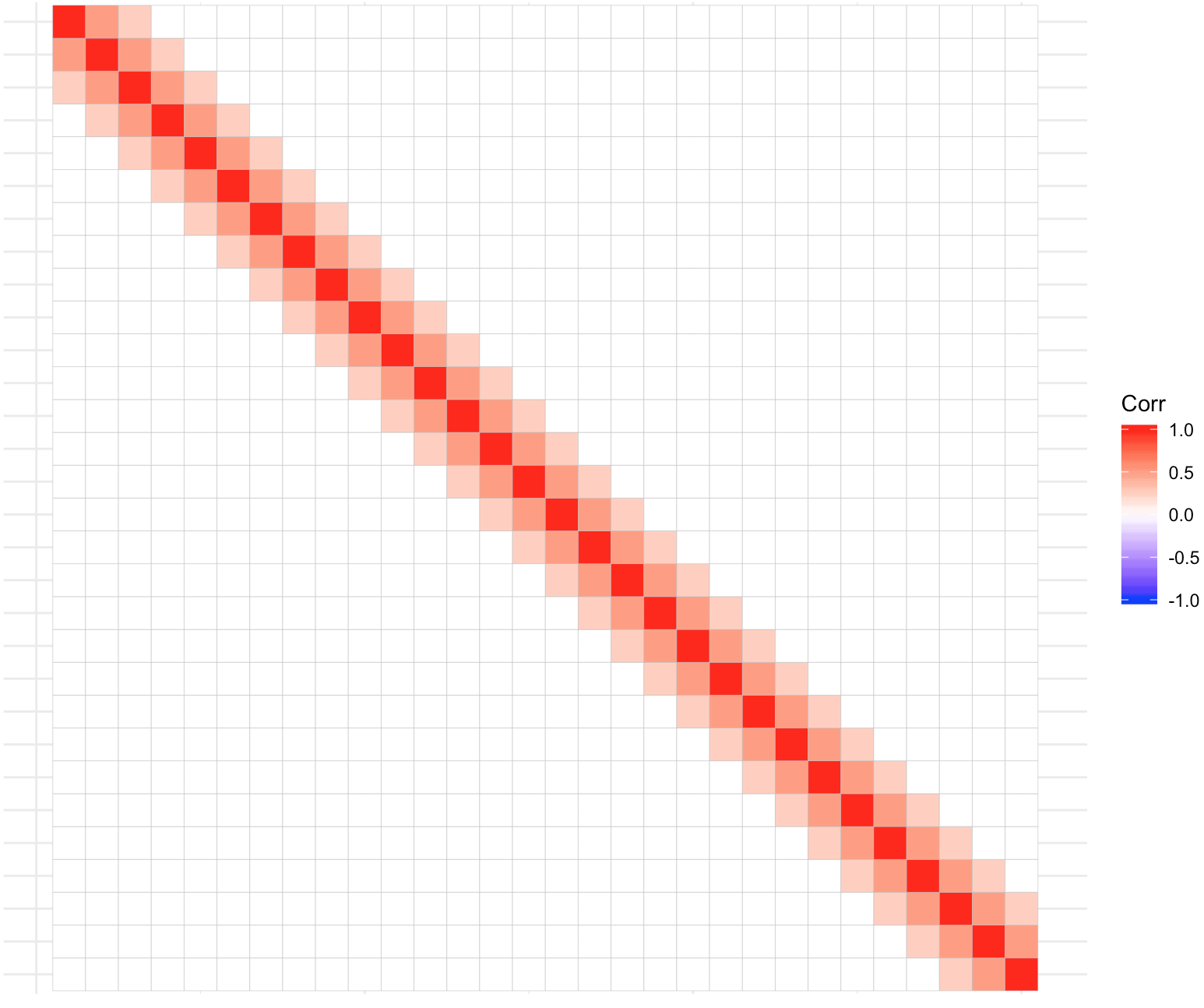}
                \caption{Precision matrix: $\Omega^\star$}
    \end{subfigure}%
    \begin{subfigure}[b]{0.5\textwidth}
                    \centering
                \includegraphics[width=0.60\textwidth]{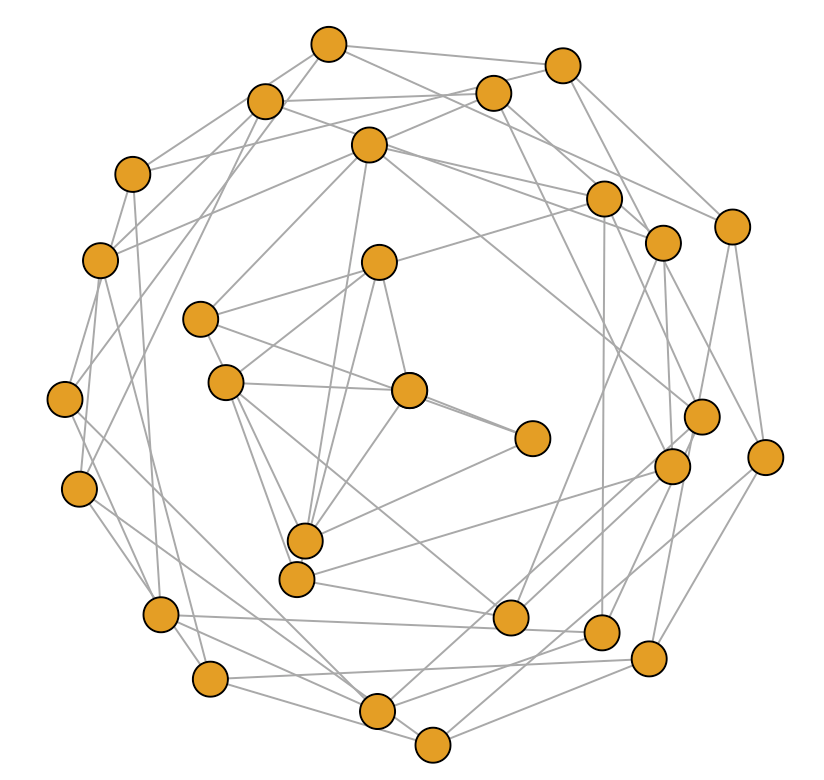}
                \caption{Graph: $G^\star$}
    \end{subfigure}%
    \caption{Precision matrix of Model 2 and its induced graph, $p=30$}
\label{fig:ar2}
\end{figure}

\begin{figure}[p]
\centering
    \begin{subfigure}[b]{0.5\textwidth}
                \centering
                \includegraphics[width=0.60\textwidth]{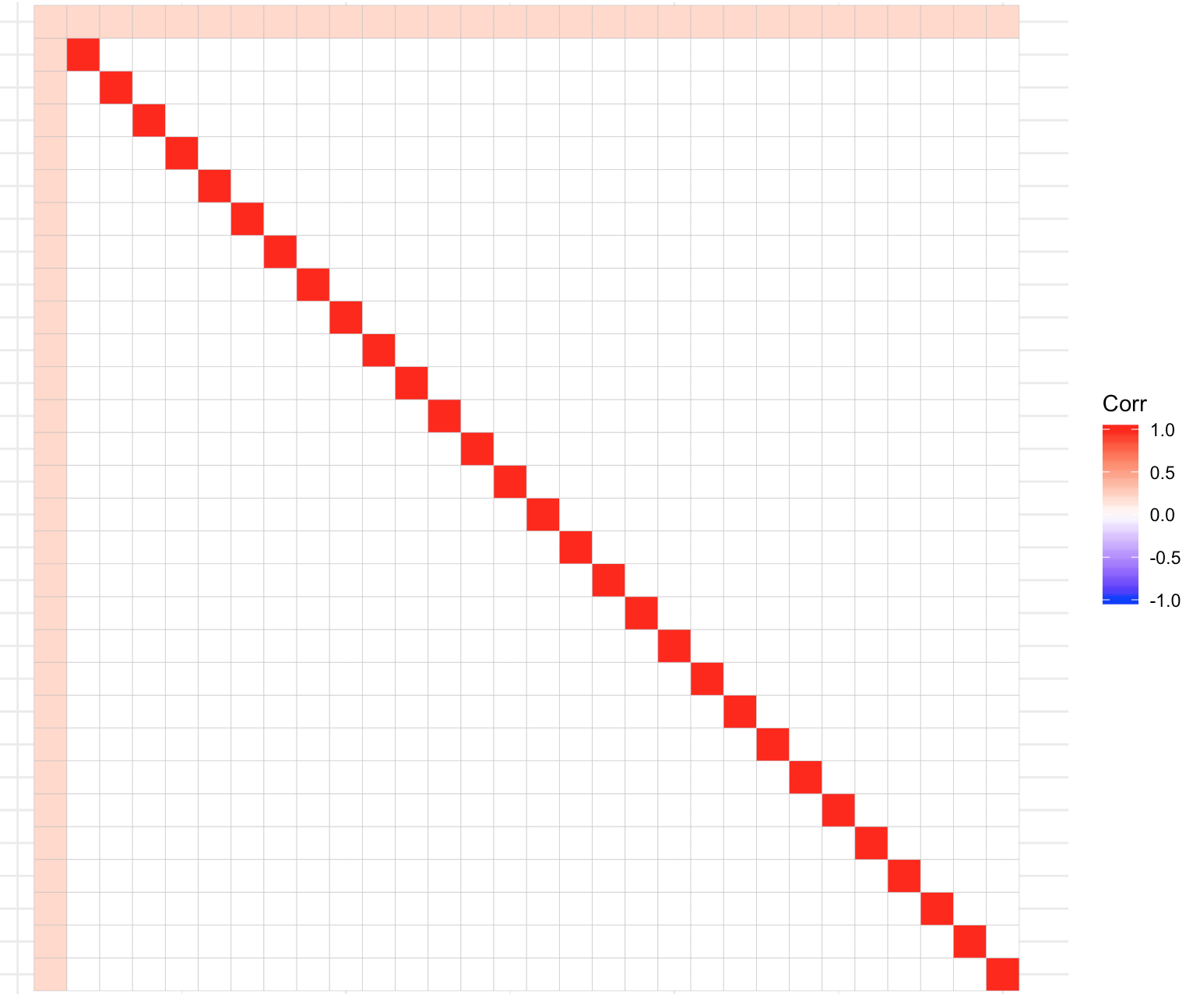}
                \caption{Precision matrix: $\Omega^\star$}
    \end{subfigure}%
    \begin{subfigure}[b]{0.5\textwidth}
                    \centering
                \includegraphics[width=0.60\textwidth]{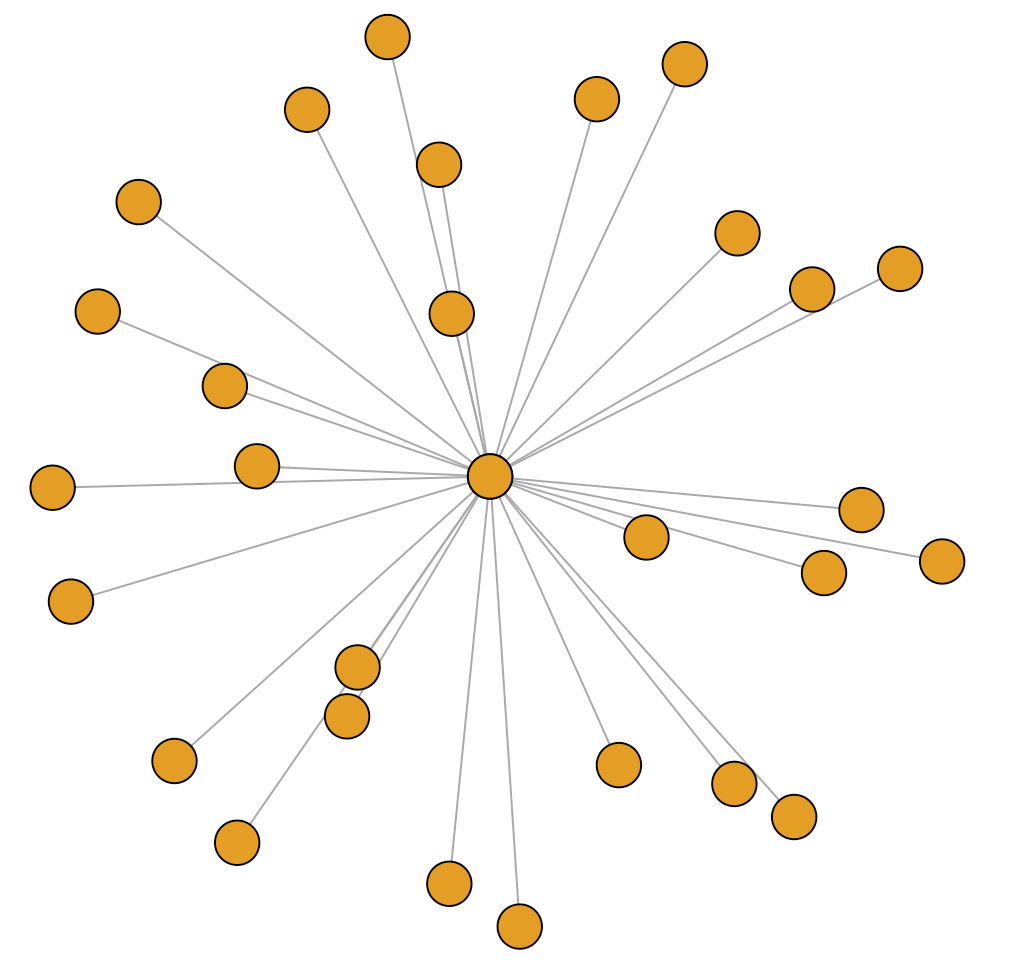}
                \caption{Graph: $G^\star$}
    \end{subfigure}%
    \caption{Precision matrix of Model 3 and its induced graph, $p=30$}
\label{fig:star}
\end{figure}

\begin{figure}[t]
\centering
    \begin{subfigure}[b]{0.5\textwidth}
                \centering
                \includegraphics[width=0.60\textwidth]{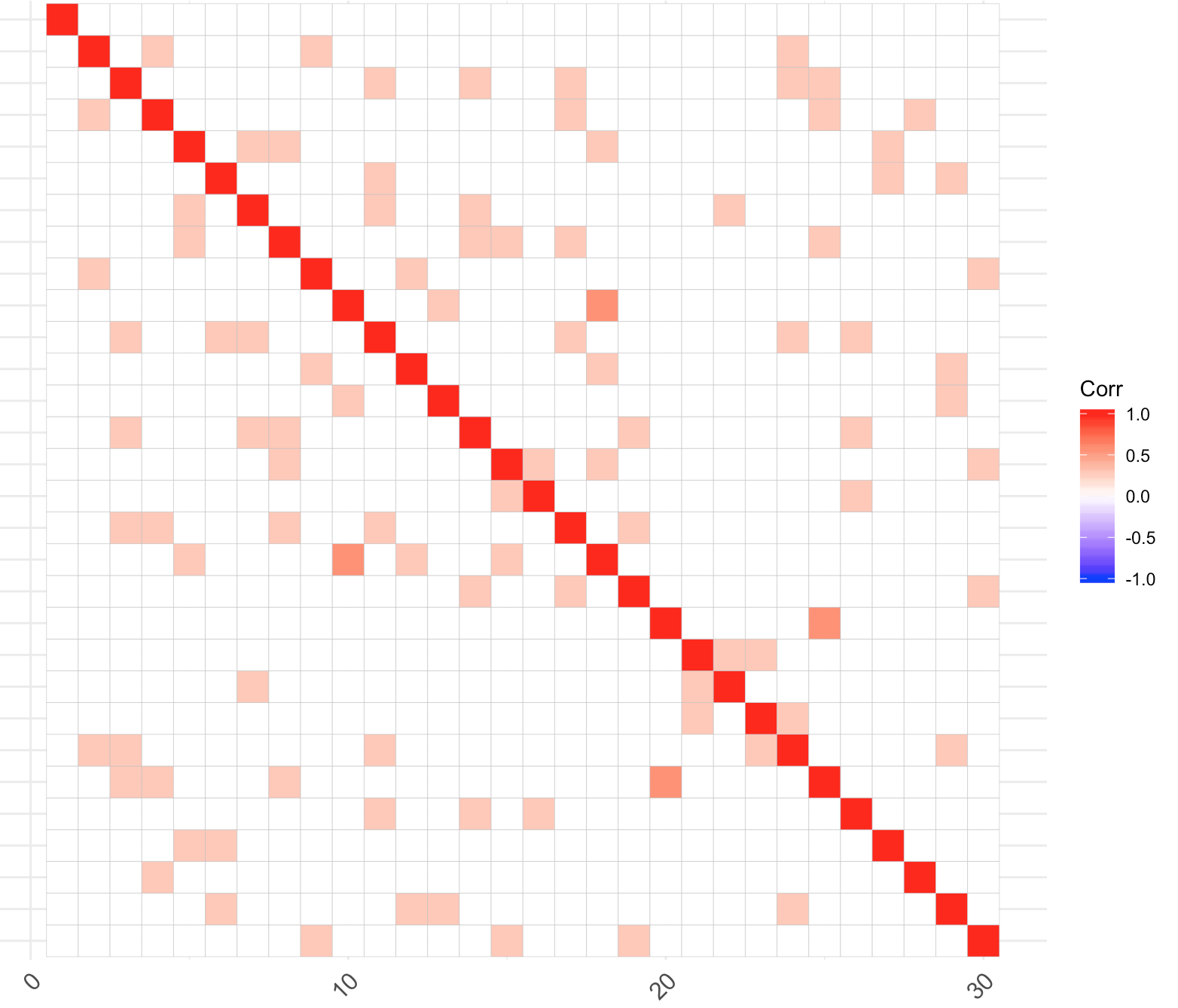}
                \caption{Precision matrix: $\Omega^\star$}
    \end{subfigure}%
    \begin{subfigure}[b]{0.5\textwidth}
                    \centering
                \includegraphics[width=0.60\textwidth]{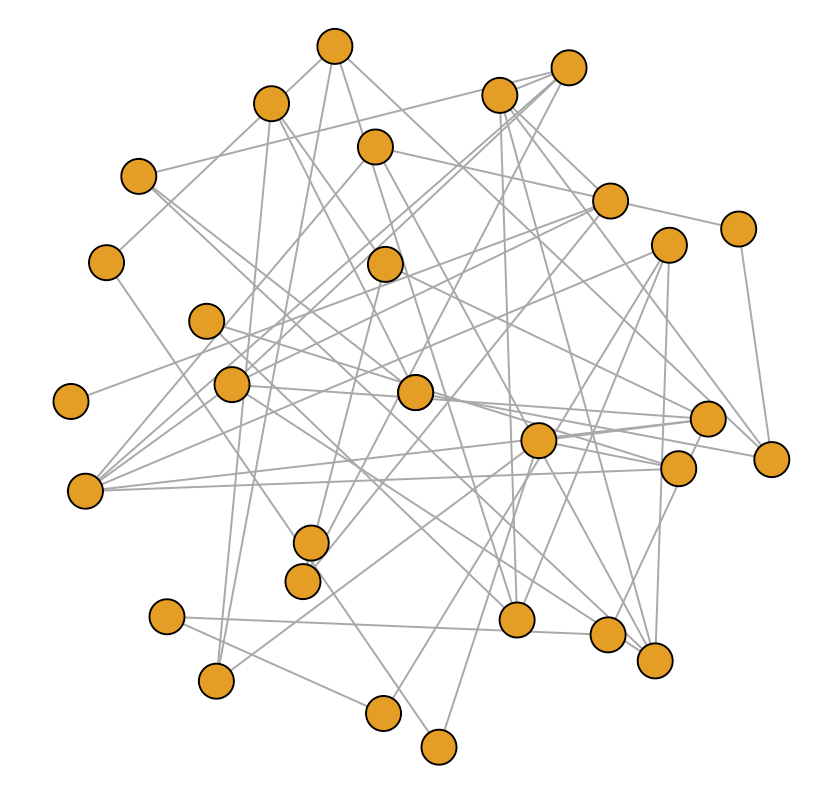}
                \caption{Graph: $G^\star$}
    \end{subfigure}%
    \caption{Precision matrix of Model 4 and its induced graph, $p=30$}
\label{fig:random1}
\end{figure}

\subsection{Results}

Adopting the criterion suggested by \citet{fan2009network},  we compute the average specificity (SP), sensitivity (SE) and Matthews correlation coefficient (MCC) across $100$ replications for each model in each setting to assess the ability of identifying the true sparsity structure of the precision matrix. Here SP, SE, and MCC are defined as
\begin{align*} 
\text{SP}  &=\frac{\text{TN}}{\text{TN}+\text{FP}}, \\
\text{SE}  & =\frac{\text{TP}}{\text{TP}+\text{FN}},\;\\
\text{MCC} &=\frac{\text{TP} \times \text{TN}-\text{FP} \times  \text{FN}}{\sqrt{(\text{TP}+\text{FP})(\text{TP}+\text{FN})(\text{TN}+\text{FP})(\text{TN}+\text{FN})}},
\end{align*}
where TP, TN, FP, and FN denote the true positives, true negatives, false positives and false negatives in the selected model, respectively. Basically, larger value of SP, SE, and MCC indicate a more accurate and robust estimator. In particular, Matthews correlation coefficient takes into account true and false positives and negatives and is generally regarded as a balanced measure which can be used even if the classes are of very different sizes; see \citet{boughorbel2017optimal}. For a comprehensive explanation of these criteria, the reader is referred to \citet{fan2009network}. The results are reported in Tables~\ref{tab:tab1}--\ref{tab:tab4}. The values in parentheses is the standard error associated with the estimator. Here,  ``GL" stands for graphical lasso approach, ``BGL" for Bayesian graphical lasso, ``GW" for using $G$-Wishart prior with scale matrix being identity, and ``eGW$_\delta$" for the proposed empirical $G$-Wishart method with shape parameter $\delta$.

For Models~\ref{model:model1}, \ref{model:model2} and \ref{model:model4}, in terms of specificity and Matthews correlation coefficient, our empirical $G$-Wishart methods ($\text{eGW}_{\delta}$) outperform all its competitors. In relatively lower-dimensional settings, e.g., $p=30$ and $50$, advantages of $\text{eGW}_{\delta}$ is even clearer. For instance, $\text{eGW}_{\delta}$ beats the other methods  in almost all of the three measures for Models~\ref{model:model1}, \ref{model:model2} and \ref{model:model4}. And for Model~\ref{model:model3}, which apparently is a difficult problem, when $p=30$, $\text{eGW}_{\delta}$ performs slightly better than the others, especially Bayesian graphical lasso, in specificity and Matthews correlation coefficient. Note that ordinary $G$-Wishart method (GW) also shows decent performance in relatively low-dimensional cases for Models~\ref{model:model1}--\ref{model:model4}, although not as good as $\text{eGW}_{\delta}$. But their performances deviate when dimension $p$ increases---GW suffers considerably from the high-dimensionality while $\text{eGW}_{\delta}$ stays competitive to the other methods, e.g., graphical lasso (GL), which starts showing advantages over the others when $p=100$. From a theoretical point of view, we have reasons to believe that this improvement of  $\text{eGW}_{\delta}$ from ordinary GW is due to the empirical prior centering. Another key observation is that the performance of $\text{eGW}_{\delta}$ basically stays stable when changing $\delta$ from $4$ to $10$, whereas sensitivity seems to gain a bit from the increase of $\delta$ while specificity suffers a little bit. Given that Matthews correlation coefficient is a balanced measure of specificity and sensitivity, its change due to the increase of $\delta$ appears to be dependent on the specific underlying graph.

\begin{table}[t]
    \centering
{\scriptsize  
\begin{tabular}{ c|c c c|c c c|c c c } 
\hline
      \multirow{2}{*}{Method} &
      \multicolumn{3}{c|}{$p=30$} &
      \multicolumn{3}{c|}{$p=50$} &
      \multicolumn{3}{c}{$p=100$} \\
      
    & SP& SE & MCC & SP& SE & MCC& SP& SE & MCC\\
    \hline
GL &   0.9312&  0.9686&  0.7374  &  0.9174 & \textbf{1.0000} &  0.6309 & 0.9443 & \textbf{1.0000} &  0.5805\\ 
&   (0.0015)& (0.0026)& (0.0041) & (0.0010) &(0.0000) & (0.0022) & (0.0005) &(0.0000) & (0.0019)\\
\hline
 BGL & 0.9981 & 0.9932 & 0.9868 & 0.9978 & 0.9393 & 0.9572&  \textbf{0.9991} & 0.9145 & 0.9327\\ 
&  (0.0002)& (0.0012)& (0.0013) & (0.0018) &(0.0008) &(0.0025) & (0.0000) &(0.0013) &(0.0010)\\
\hline
GW  & 0.9898 & 0.9995 & 0.9521 & 0.9736 & 0.9808 & 0.8180 & 0.9567 & 0.7879 & 0.5128\\ 
& (0.0009)& (0.0005)& (0.0039) & (0.0007)& (0.0020)& (0.0039) & (0.0004)& (0.0042)& (0.0033)\\
\hline
 eGW$_4$ & \textbf{0.9982} & \textbf{1.0000} & \textbf{0.9910} & \textbf{0.9991} & 0.9999 & \textbf{0.9923} & \textbf{0.9991} & 0.9984 & \textbf{0.9844}\\ 
& (0.0003)& (0.0000)& (0.0016) & (0.0001)& (0.0001)& (0.0011) & (0.0001)& (0.0005)& (0.0013)\\
\hline
 eGW$_{10}$ & 0.9931 & \textbf{1.0000} & 0.9668 & 0.9957 & \textbf{1.0000} & 0.9663 & 0.9971 & 0.9985 & 0.9542\\ 
& (0.0007)& (0.0000)& (0.0033) & (0.0004)& (0.0000)& (0.0026) & (0.0001)& (0.0004)& (0.0021)\\
\hline
\end{tabular}
}
 \caption{Model 1: specificity (SP), sensitivity (SP) and Matthews correlation coefficient (MCC) across 100 replications with $p=30, 50, 100$ and $n=100$. 
}
 \label{tab:tab1}
\end{table}

\begin{table}[t]
\centering
{\scriptsize 
\begin{tabular}{ c|c c c|c c c|c c c } 
\hline
      \multirow{2}{*}{Method} &
      \multicolumn{3}{c|}{$p=30$} &
      \multicolumn{3}{c|}{$p=50$} &
      \multicolumn{3}{c}{$p=100$} \\
    & SP& SE & MCC & SP& SE & MCC& SP& SE & MCC\\
    \hline
GL &   0.9449&  0.6812& 0.6358 & 0.9546 &  0.6781 & 0.6088 & 0.9632 & \textbf{0.6827} & 0.5541\\
&   (0.0016)& (0.0035)& (0.0035) & (0.0008) &(0.0025) & (0.0028) & (0.0004) &(0.0019) & (0.0019)\\
\hline
 BGL & 0.9471 & 0.5743 & 0.6360 & 0.9836 & 0.4852 & 0.5776&  \textbf{0.9990} & 0.4316 & 0.5492\\ 
&  (0.0015)& (0.0035)& (0.0035) & (0.0007) &(0.0045) &(0.0035) & (0.0000) &(0.0022) &(0.0019)\\
\hline
GW  & 0.9863 & 0.8836 & 0.8865 & 0.9727 & 0.8102 & 0.7623 & 0.9648 & 0.5876 & 0.5387\\ 
& (0.0008)& (0.0057)& (0.0045) & (0.0007)& (0.0050)& (0.0045) & (0.0004)& (0.0030)& (0.0028)\\
\hline
eGW$_4$ & \textbf{0.9899} & 0.8781 & \textbf{0.8938} & \textbf{0.9939} & 0.8306 & \textbf{0.8703} & 0.9968 & 0.6177 & \textbf{0.7395}\\ 
& (0.0007)& (0.0063)& (0.0053) & (0.0003)& (0.0066)& (0.0051) & (0.0001)& (0.0042)& (0.0029)\\
\hline
eGW$_{10}$ & 0.9697 & \textbf{0.9231} & 0.8659 & 0.9849 & \textbf{0.8997} & 0.8701 & 0.9924 & 0.6062 & 0.6868\\ 
& (0.0013)& (0.0053)& (0.0047) & (0.0007)& (0.0038)& (0.0033) & (0.0002)& (0.0054)& (0.0043)\\
\hline
\end{tabular}
}
\caption{Model 2: See caption of Table~\ref{tab:tab1}. 
} \label{tab:tab2}
\end{table}

\begin{table}[t]
\centering
{\scriptsize 
\begin{tabular}{ c|c c c|c c c|c c c } 
\hline
      \multirow{2}{*}{Method} &
      \multicolumn{3}{c|}{$p=30$} &
      \multicolumn{3}{c|}{$p=50$} &
      \multicolumn{3}{c}{$p=100$} \\
    & SP& SE & MCC & SP& SE & MCC& SP& SE & MCC\\
    \hline
GL &   0.9889&  0.3550& 0.3745 & 0.9546 &  0.3894 & 0.6097 & 0.9503 & \textbf{1.0000}& 0.8852\\
&   (0.0000)& (0.0022)& (0.0018) & (0.0003) &(0.0020) & (0.0004) & (0.0007) &(0.0000) & (0.0008)\\
\hline
 BGL & 0.9979 & \textbf{0.4413} &  0.5642 & 0.9933 & \textbf{0.5967} & \textbf{0.6229} &  0.9931 & \textbf{1.0000} & \textbf{0.9378}\\ 
&  (0.0000)& (0.0009)& (0.0006) & (0.0000) &(0.0003) &(0.0003) & (0.0000) &(0.0000) &(0.0001)\\
\hline
GW  & 0.9787 & 0.4036 & 0.4852 & 0.9624 &  0.4218 & 0.3816 & 0.9346 & 0.4707 &  0.2581\\ 
& (0.0009)& (0.0050)& (0.0070) & (0.0007)& (0.0052)& (0.0049) & (0.0004)& (0.0027)& (0.0021)\\
\hline
eGW$_4$ & \textbf{0.9995} & 0.3459 & 0.5645 & \textbf{0.9999} &  0.3419 & 0.5719 & \textbf{0.9960} & 0.9416 &  0.9069\\ 
& (0.0001)& (0.0010)& (0.0016) & (0.0000)& (0.0009)& (0.0010) & (0.0001)& (0.0024)& (0.0032)\\
\hline
eGW$_{10}$ & 0.9991 & 0.3500 & \textbf{0.5647} & 0.9995 &  0.3465 & 0.5692 & 0.9930 & 0.9346 &  0.8573\\ 
& (0.0002)& (0.0020)& (0.0020) & (0.0001)& (0.0017)& (0.0018) & (0.0002)& (0.0035)& (0.0043)\\
\hline
\end{tabular}
}
\caption{Model 3: See caption of Table~\ref{tab:tab1}.
} 
\label{tab:tab3}
\end{table}

\begin{table}[t]
\centering
{\scriptsize
\begin{tabular}{ c|c c c|c c c|c c c } 
\hline
      \multirow{2}{*}{Method} &
      \multicolumn{3}{c|}{$p=30$} &
      \multicolumn{3}{c|}{$p=50$} &
      \multicolumn{3}{c}{$p=100$} \\
    & SP& SE & MCC & SP& SE & MCC& SP& SE & MCC\\
    \hline
GL &    0.9274&  0.8580& 0.7070 & 0.9145 &  \textbf{0.7306} &  0.5762 &0.9288 &  \textbf{0.4929}&  0.4093\\
&   (0.0015)& (0.0035)& (0.0039) & 0.0010) &(0.0033) & (0.0026) & (0.0008) &(0.0030) & (0.0024)\\
\hline
BGL & 0.9729 & 0.6588 & 0.6899 & 0.9891 & 0.3487 &  0.5003 &  0.9855 & 0.2132 & 0.3329\\ 
&  (0.0086)& (0.0028)& (0.0060) & (0.0005) &(0.0038) &(0.0031) & (0.0000) &(0.0000) &(0.0001)\\
\hline
GW & 0.9804 & 0.8647& 0.8508 & 0.9616 &  0.6455 & 0.6669 & 0.9428 &0.3018 &  0.3074\\
& (0.0010)& (0.0059)& (0.0059) & (0.0008)& (0.0049)& (0.0050) & (0.0005)& (0.0021)& (0.0027)\\
\hline
eGW$_4$ & \textbf{0.9920} & 0.8476 & \textbf{0.8794} & \textbf{0.9906} &  0.5956 & 0.7043 & \textbf{0.9910} &0.2537 &  0.4114\\
& (0.0008)& (0.0064)& (0.0049) & (0.0004)& (0.0054)& (0.0043) & (0.0002)& (0.0020)& (0.0026)\\
\hline
eGW$_{10}$ & 0.9743 & \textbf{0.9098} & 0.8613 & 0.9754 &  0.6930 & \textbf{0.7134} & 0.9829 &0.3020 &  \textbf{0.4150}\\ 
& (0.0012) & (0.0055)& (0.0053) & (0.0009)& (0.0051)& (0.0045) & (0.0003)& (0.0022)& (0.0025)\\
\hline
\end{tabular}
}
\caption{Model 4: See caption of Table~\ref{tab:tab1}.
} \label{tab:tab4}
\end{table}

\section{Real data illustration}
\label{S:real_data}
Here we use our method based on an empirical $G$-Wishart prior to investigate gene regulatory network rewriting for breast cancer. The data set originally comes from TCGA (version: May 6 2017) \citep{cancer2012comprehensive} and is used by \citet{zhang2017diffgraph}, in which they performed a pathway-based analysis and only
focused on genes in the breast cancer pathway collected from the Kyoto Encyclopedia of Genes and Genomes database \citep{kanehisa2000kegg}.  A detailed subscription of this data set can be found in \citet{zhang2017diffgraph}. Here we only consider luminal A subtype breast cancer and use standardized mRNA expression level (Agilent G450 microarray). The number of patients with luminal A breast cancer is 207 and the number of genes investigated is 139. 

\begin{figure}[t]
\centering
    \includegraphics[width=0.75\textwidth]{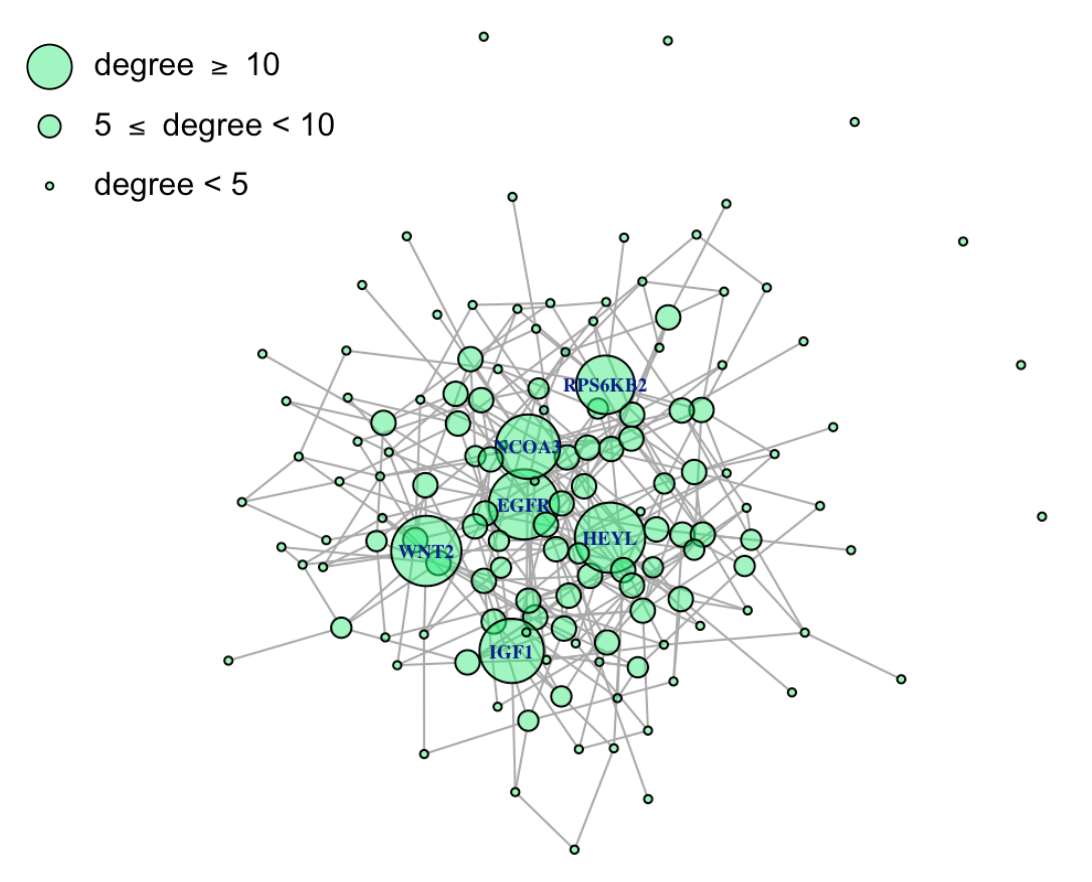}
                \caption{Gene regulatory network of luminal A breast cancer. The node size is proportional to the node's degree.}
\label{fig:gene_network}
\end{figure}

We apply our method to the selected data set and obtain MCMC samples using the sampling scheme described in Section~\ref{s:simulation}. An edge is selected if its posterior inclusion probability is greater than 0.5. The median probability model (MPM) gene regulatory network is displayed in Figure~\ref{fig:gene_network}.  In a gene regulatory network for diseases, the genes with most connections with other genes, i.e., highest degrees, often play a key role in the disease pathway and are of top interests to biologists and clinicians. These kind of genes are usually called ``hub genes.''  To evaluate the performance of our method, we identify the hub genes of luminal A breast cancer based on the posterior distribution of each node's degree and its relative rank among the 139 genes. We selected four top hub genes: EGFR, NCOA3, RPS6KB2 (S6K2) and IGF1. Posterior distributions of their degree and their rank of degree are displayed in Figures~\ref{fig:EGFR}--\ref{fig:IGF1}.  A key observation is that the posterior distributions for degree and rank of degree are all unimodal, and with most of its mass concentrated around the mode. Through a literature review, we discovered that the four genes selected are all well-known breast cancer target genes/hub genes and have strong correlation with breast cancer risk and prognosis. EGFR, an important signature of breast cancer molecular subtypes, has been shown to regulate epithelial-mesenchymal transition, migration, and tumor invasion \citep{masuda2012role}. Nowadays, several anti-EGFR therapies have been approved for breast cancer. In addition, NCOA3 is overexpressed in a large proportion of primary human breast tumors, and high levels of NCOA3 expression are associated with worse survival rate \citep{burwinkel2005association}. IGF1 is a point of convergence for major signaling pathways implicated in breast cancer growth. Overexpression of IGF1/IGF1R is a well-studied breast cancer risk factor.

\begin{figure}[p]
\centering
    \begin{subfigure}[b]{0.5\textwidth}
                \centering
                \includegraphics[width=0.60\textwidth]{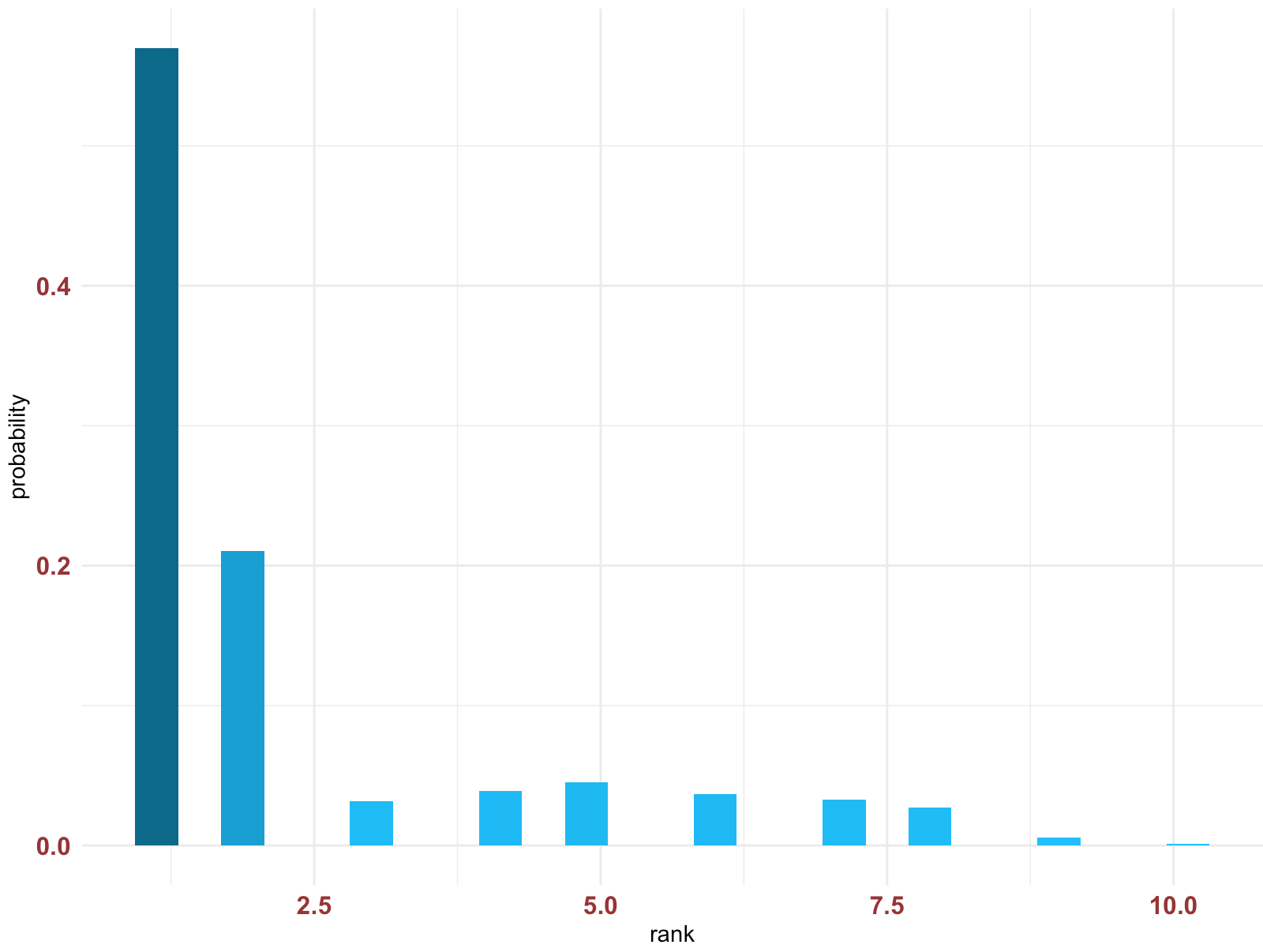}
                \caption{Posterior distribution of rank}
    \end{subfigure}%
    \begin{subfigure}[b]{0.5\textwidth}
                    \centering
                \includegraphics[width=0.60\textwidth]{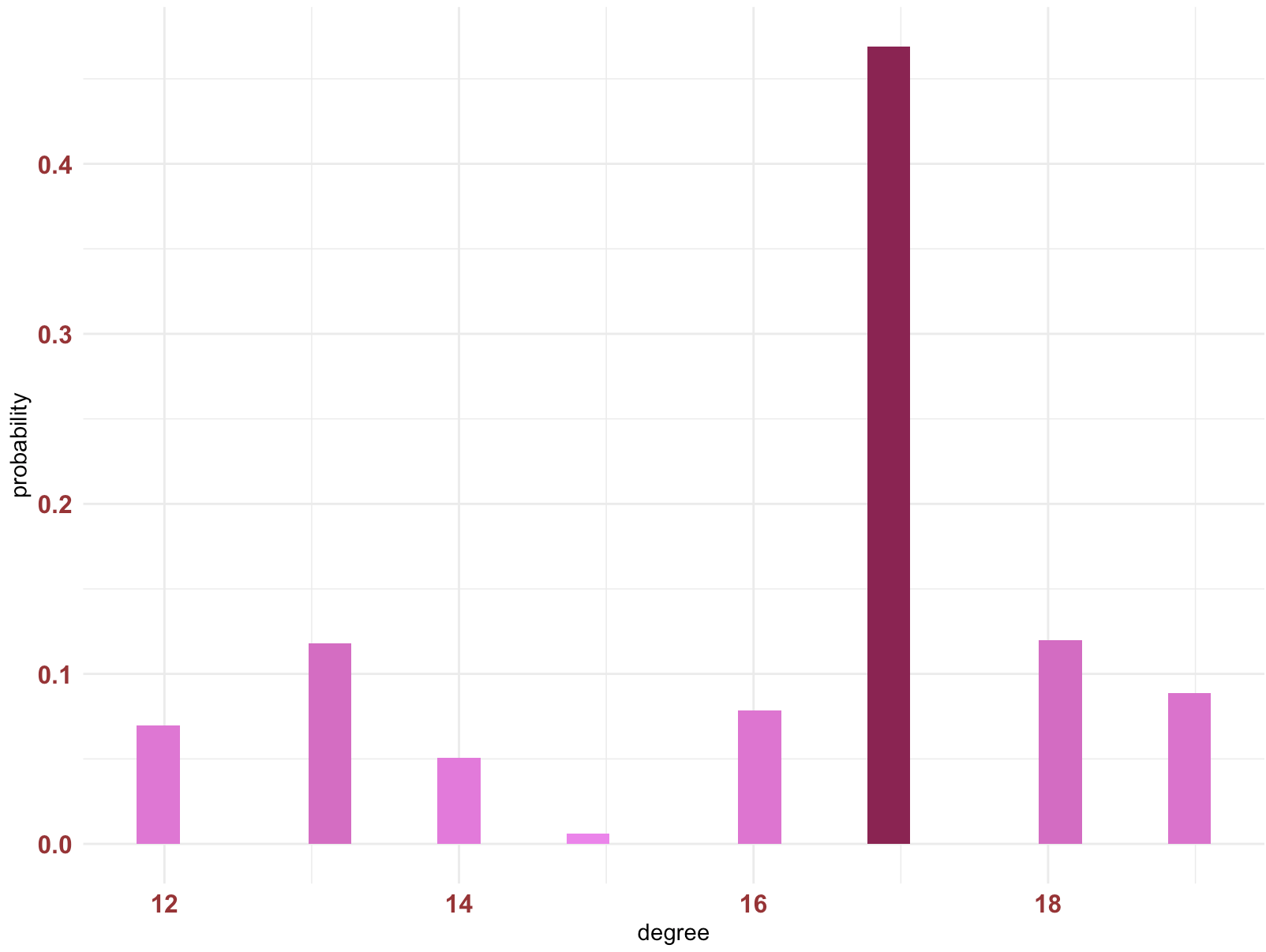}
                \caption{Posterior distribution of degree}
    \end{subfigure}%
    \caption{Gene EGFR}
\label{fig:EGFR}
\end{figure}
\begin{figure}[p]
\centering
    \begin{subfigure}[b]{0.5\textwidth}
                \centering
                \includegraphics[width=0.60\textwidth]{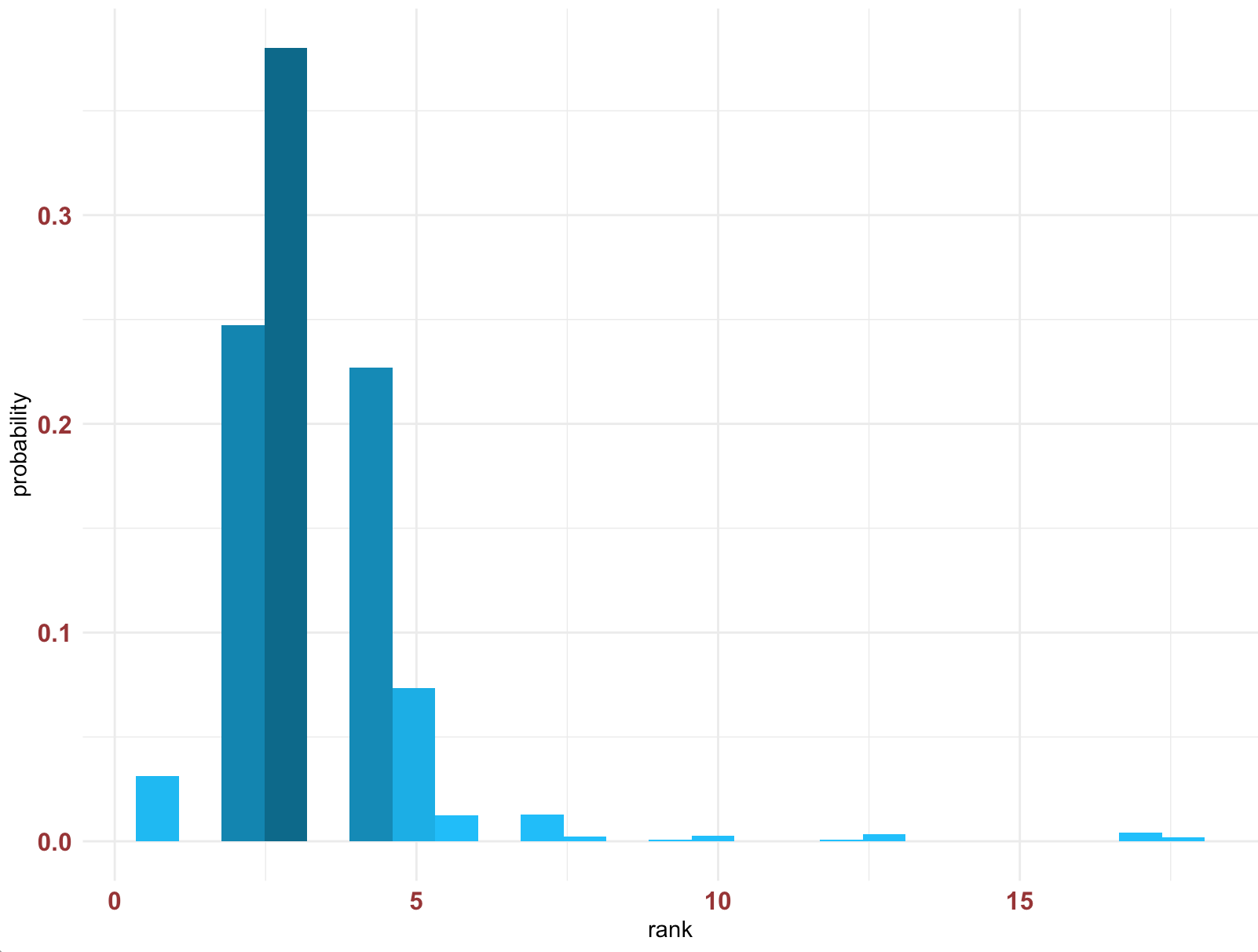}
                \caption{Posterior distribution of rank}
    \end{subfigure}%
    \begin{subfigure}[b]{0.5\textwidth}
                    \centering
                \includegraphics[width=0.60\textwidth]{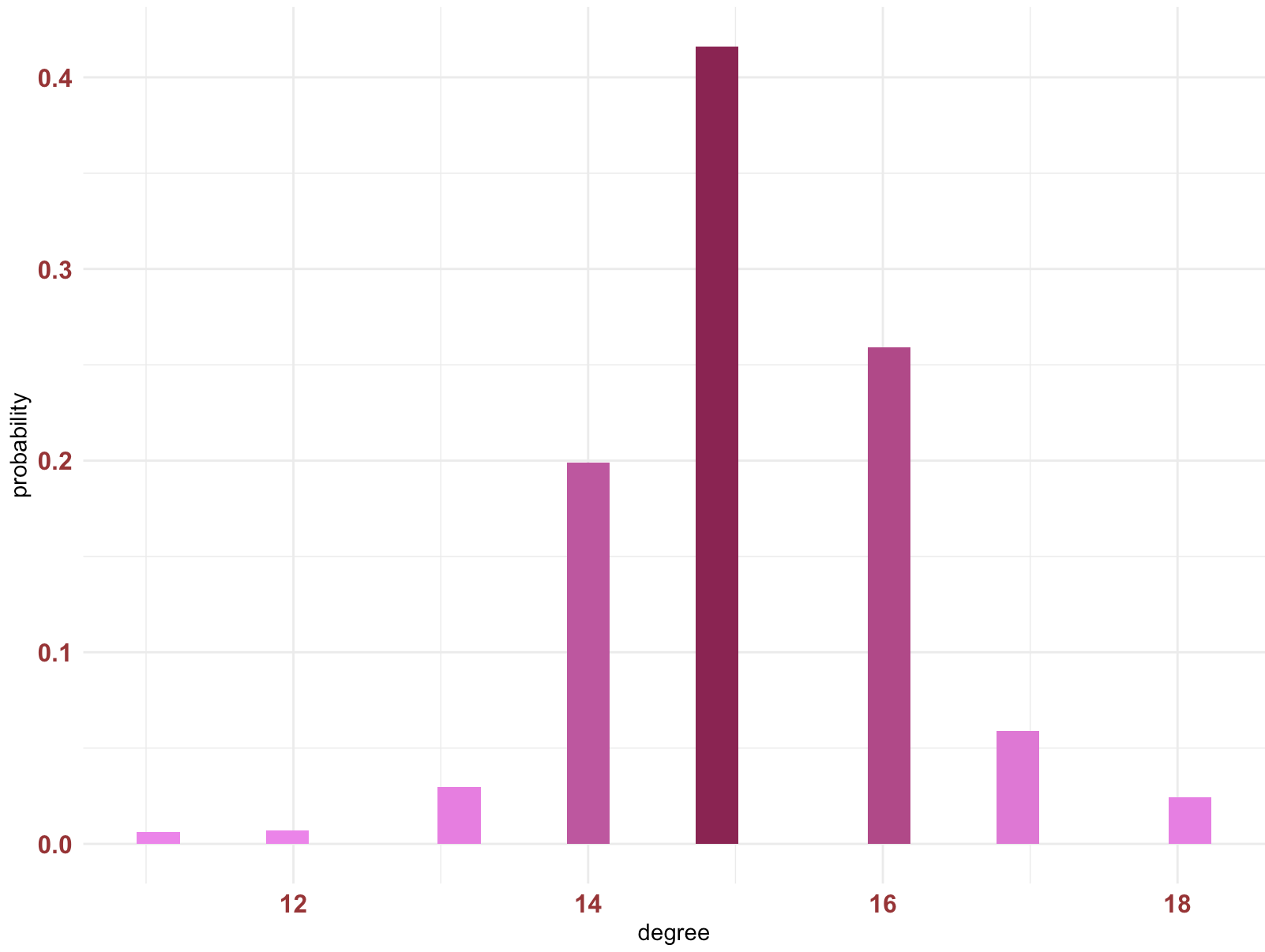}
                \caption{Posterior distribution of degree}
    \end{subfigure}%
    \caption{Gene RPS6KB2 (S6K2)}
\label{fig:RPS6KB2}
\end{figure}
\begin{figure}[p]
\centering
    \begin{subfigure}[b]{0.5\textwidth}
                \centering
                \includegraphics[width=0.60\textwidth]{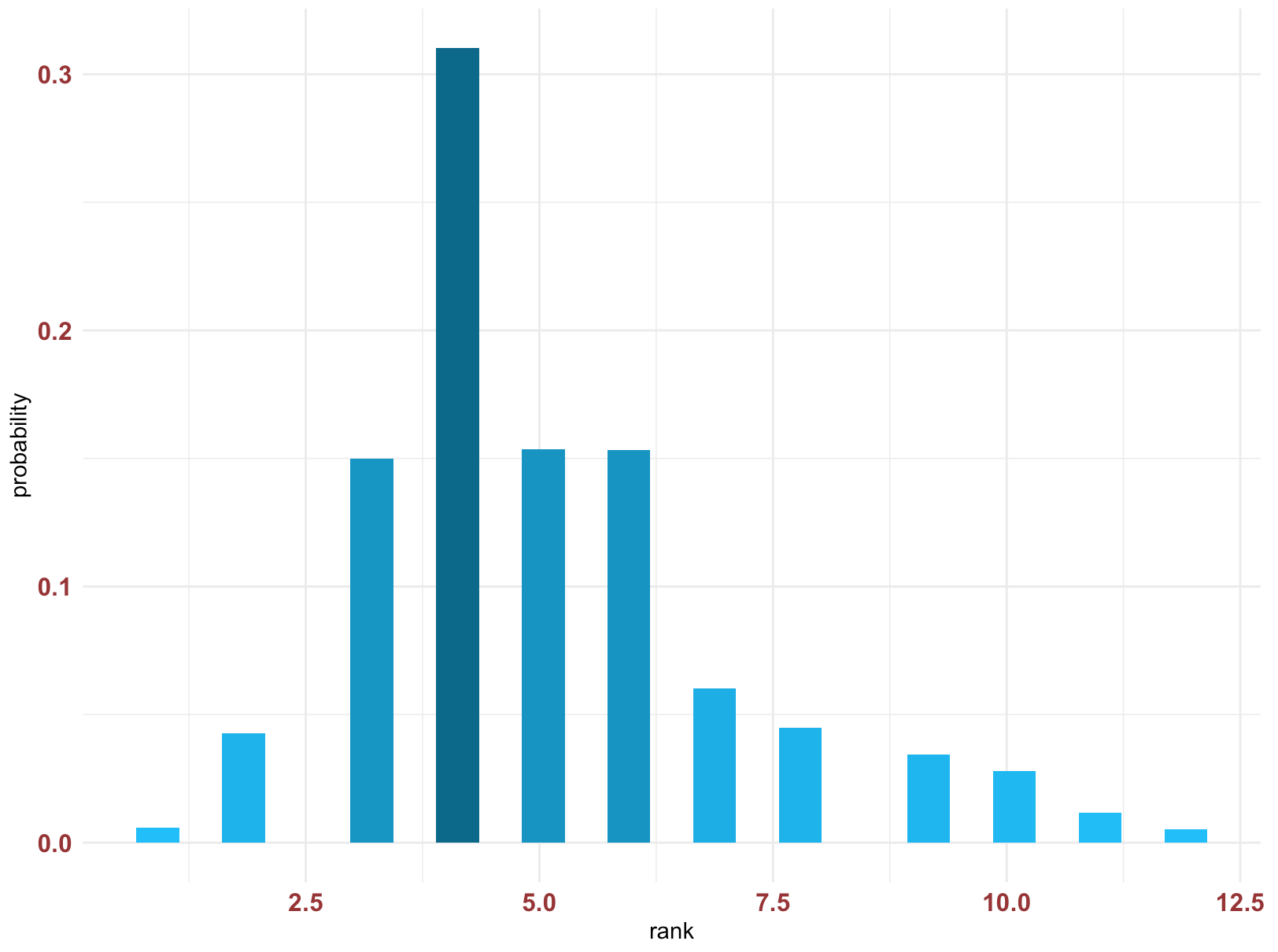}
                \caption{Posterior distribution of rank}
    \end{subfigure}%
    \begin{subfigure}[b]{0.5\textwidth}
                    \centering
                \includegraphics[width=0.60\textwidth]{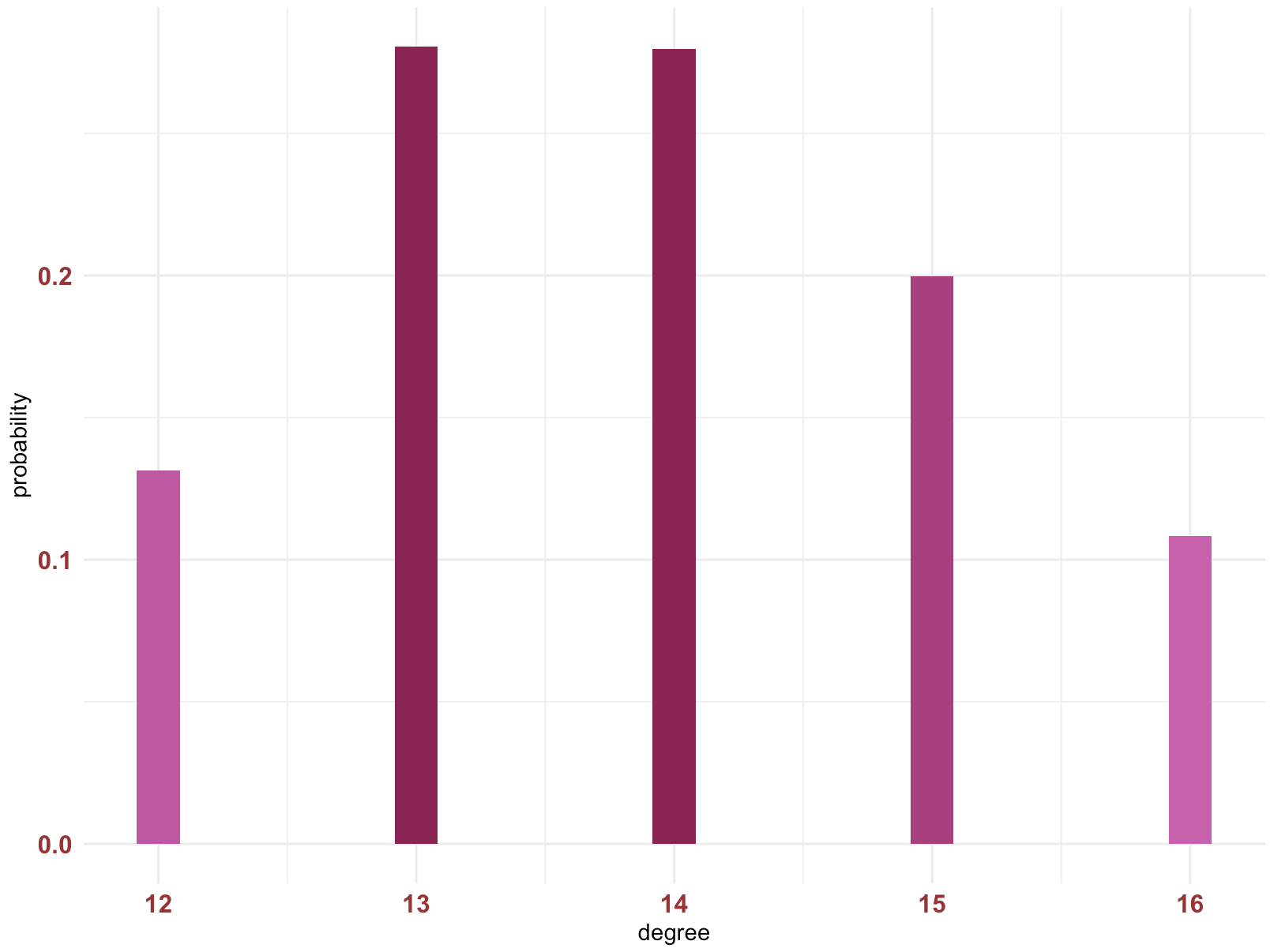}
                \caption{Posterior distribution of degree}
    \end{subfigure}%
    \caption{Gene NCOA3}
\label{fig:NCOA3}
\end{figure}
\begin{figure}[p]
\centering
    \begin{subfigure}[b]{0.5\textwidth}
                \centering
                \includegraphics[width=0.60\textwidth]{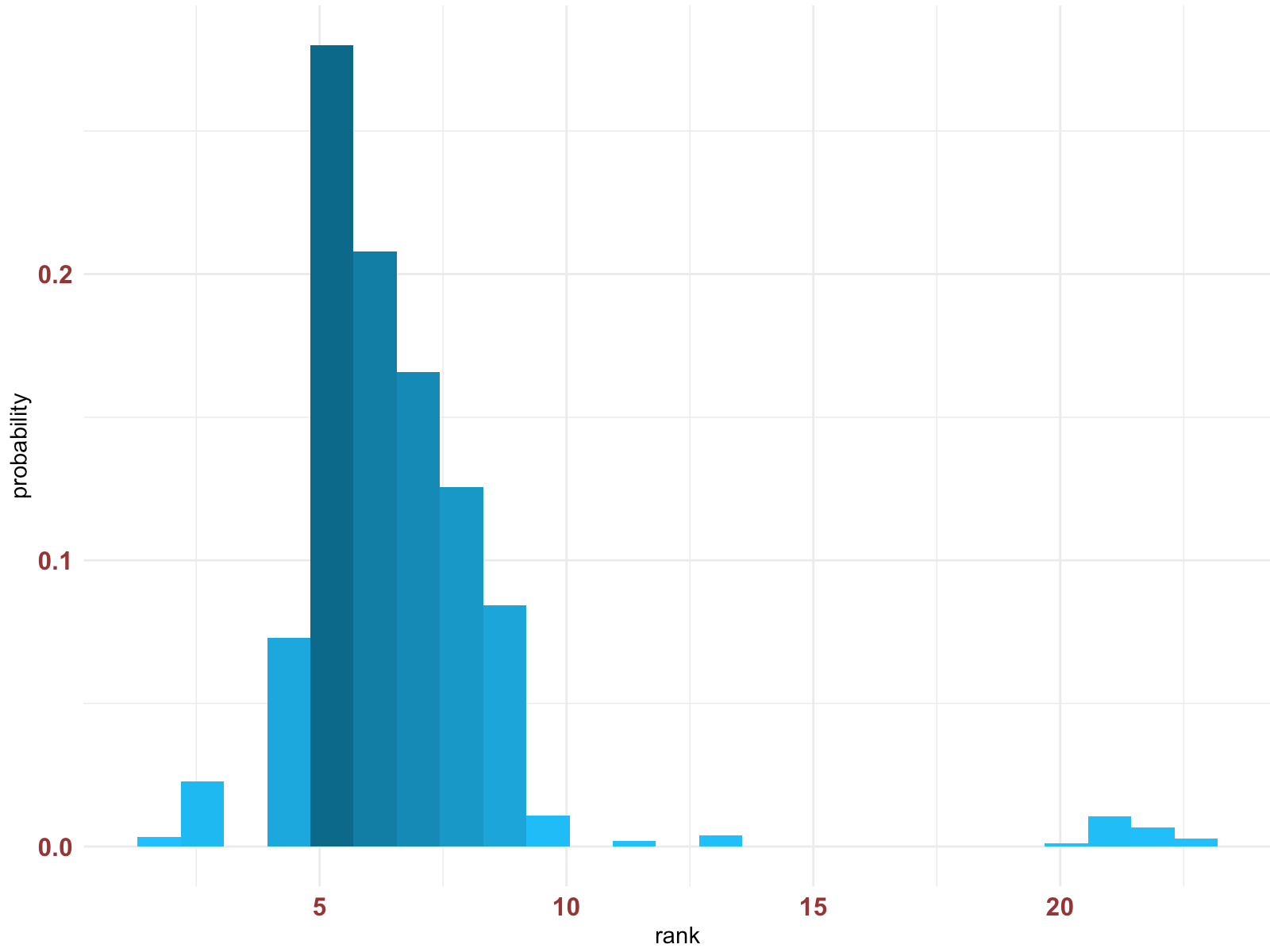}
                \caption{Posterior distribution of rank}
    \end{subfigure}%
    \begin{subfigure}[b]{0.5\textwidth}
                    \centering
                \includegraphics[width=0.60\textwidth]{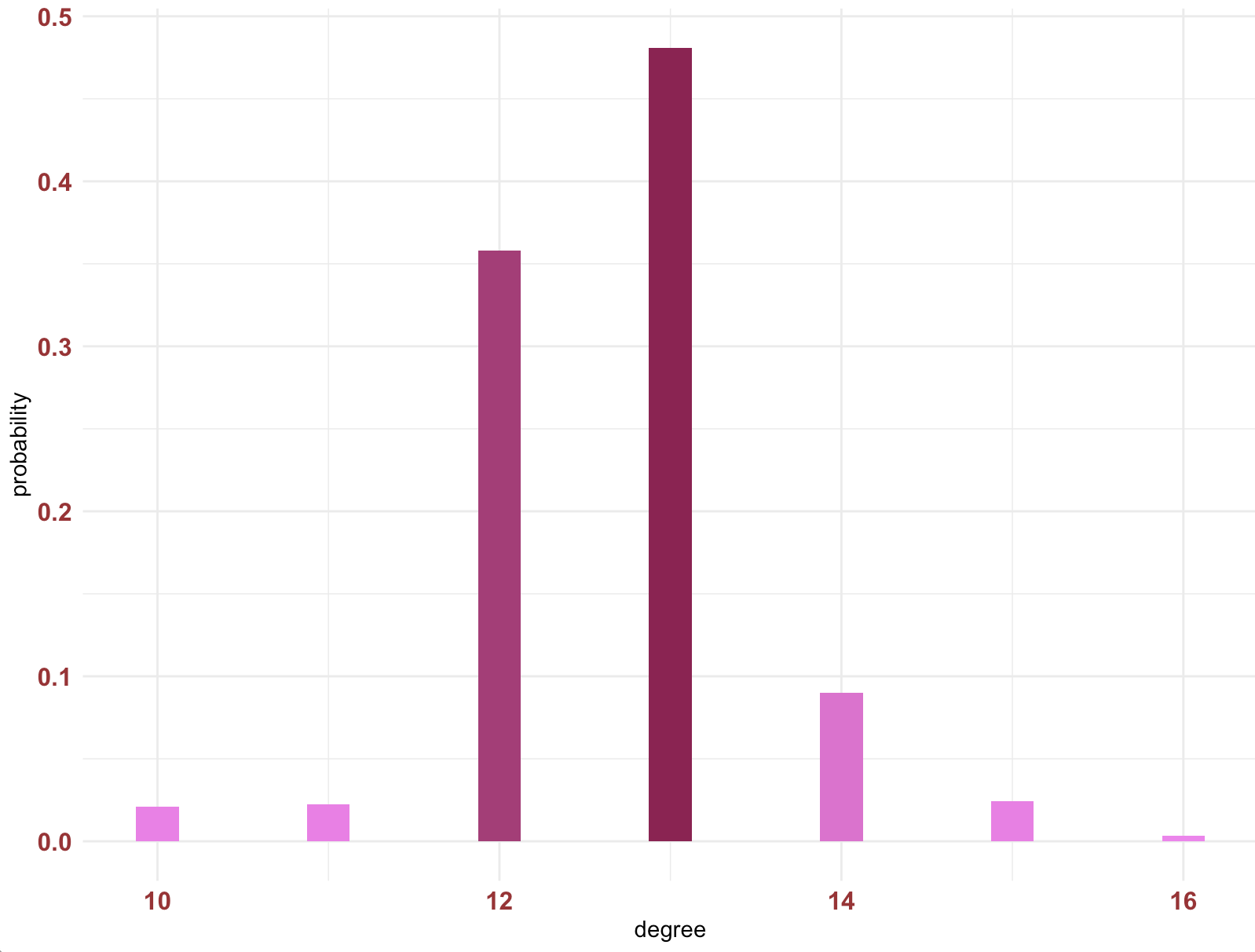}
                \caption{Posterior distribution of degree}
    \end{subfigure}%
    \caption{Gene IGF1}
\label{fig:IGF1}
\end{figure}

\section{Concluding remarks}
\label{S:discuss}

In this paper, we proposed a hierarchical prior framework combining a prior for graph and an empirical $G$-Wishart prior for precision matrix to implement Gaussian graphical model structure learning. Theoretically, under mild conditions, the induced posterior contraction rate---in terms of Frobenius distance on precision matrices---agrees with that of the graphical lasso. Using a Laplace approximation, we achieve fast and accurate computation of marginal posterior probabilities. In practice, our method demonstrates strong performance in graphical structure learning across various settings. Finally, we apply our method to a high-dimensional gene expression data, with interesting results.

However, there are still some open questions and further improvements on this model may be done in the future. First, we assume the underlying true graph to be decomposable only for technical reasons, namely, that certain high-dimensional integrals over $\Pp_{G^\star}$ can be handled when $G^\star$ is decomposable; our Assumption~\ref{asmp:decomposable} can be relaxed if these integrals could be worked out for non-decomposable graphs. In addition, we could have some room to let $\delta$ to be increasing with $n$, for example, $\delta=O(n)$. This provides possibility for prior normalizing constant to be approximated by Laplace method with asymptotic accuracy. However, we failed to establish our targeted posterior concentration rate when $\delta=O(n)$, again, due to the technical difficulty of calculating complex matrix integrals in establishing Lemma~\ref{lem:N_n}. With some advanced technical tools, theories about edge selection consistency may also be developed, which, as far as we know, has not yet been established in the Bayesian literature.

Some other follow-ups to this work can also be done. For example, we considered estimating sparse precision matrix in this paper, but sparse covariance matrix estimation problem can also be addressed similarly. In the future, we can also consider incorporating some specific sparsity structures, like banded or block diagonal matrices, by specifying some structural priors for graph. Other kinds of low-dimensional structure could also be considered, such as factor models.

\section*{Acknowledgments}

This work is partially supported by the National Science Foundation, DMS--1737933.

\appendix

\section{Proofs of the theorems}

\subsection{Proof of Theorem~\ref{thm:concentration}}
\label{SS:proof.concentration}

For a constant $M$ and the rate $\eps_n(G^\star)$, define the event 
\[ B_n = \{\Omega: H^2(p_{\Omega^\star}, p_\Omega) > 1 - e^{-M^2 \eps_n(G^\star)^2}\}, \]
which, for large $n$, since $\eps_n(G^\star) \to 0$, is equivalent to 
\[ \{\Omega: H(p_{\Omega^\star}, p_\Omega) > M \eps_n(G^\star)\} \]
as in the statement of Theorem~\ref{thm:concentration}.  Then $\Pi^n(B_n)$ can be written as
\begin{equation}
    \Pi^n(B_n)=\frac{\sum_{G} \pi(G) \int_{B_n \cap \Pp_G} R_n^{\alpha}(\Omega)\pi_n(\Omega \mid G) \, \diff{\Omega}}{\sum_{G} \pi(G) \int_{\Pp_G} R_n^{\alpha}(\Omega)\pi_n(\Omega \mid G) \, \diff{\Omega}}
    \label{eq:post_probability}
\end{equation}
where $R_n(\Omega)=L_n(\Omega)/L_n(\Omega^\star)$ representing the likelihood ratio, $\alpha \in (0,1)$, and the summation is over all graphs $G$. 

We first present two lemmas needed for the proof of Theorems~\ref{thm:concentration}--\ref{thm:dimension}.  For notational simplicity, we use $N_n$ and $D_n$ to represent the numerator and denominator in \eqref{eq:post_probability} respectively.  Then Lemma~\ref{lem:D_n} gives a probabilistic lower bound on $D_n$ and Lemma~\ref{lem:N_n} gives an upper bound on 
\begin{equation}
    \E_{\Omega^\star}(N_n) = \sum_{G} \pi(G) \int_{B_n \cap \Pp_G} \E_{\Omega^\star }\big[ R_n^{\alpha}(\Omega)\pi_n(\Omega \mid G)\big] \, \diff{\Omega},
    \label{eq:EN_n}
\end{equation}
where the expectation is with respect to the model $X_1,\ldots,X_n \iid \nm_p(0, \Omega^{\star-1})$.  

\begin{lem}
\label{lem:D_n}
Under Assumptions~\ref{asmp:dimension}--\ref{asmp:decomposable}, if the prior satisfies Condition~\ref{cond:P}, then 
\[\prob_{\Omega^\star}\{D_n < \pi(G^\star)e^{-(1+\delta/2)(p+|G^\star|)\log n}\} \to 0 \quad \text{as $n \to \infty$.} \]
\end{lem}

\begin{proof}
See Appendix~\ref{proof:D_n}
\end{proof}

\begin{lem}
\label{lem:N_n}
Under Assumption~\ref{asmp:dimension}, if the prior satisfies Condition~\ref{cond:P}, then there exists a constant $A=A(\alpha)$ such that 
\[\E_{\Omega^\star}(N_n) \le e^{-AM^2 n\eps_n^2(G^\star)} \sum_G \pi(G) e^{\delta(p+|G|)\log\xi_n}.\]
\end{lem}

\begin{proof}
See Appendix~\ref{proof:N_n}
\end{proof}

Now we are ready to proceed with the proof of Theorem~\ref{thm:concentration}.  If we let 
\[ d_n=\pi(G^\star) \exp\{-(1+\delta/2)(p+|G^\star|)\log n\} \]
be the lower bound in Lemma~\ref{lem:D_n}, then we have
\begin{align*}
    \Pi^n(B_n)=\frac{N_n}{D_n} \one(D_n\ge d_n)+\frac{N_n}{D_n}\one(D_n<d_n) \le \frac{N_n}{d_n} + \one(D_n<d_n).
\end{align*}
Taking expectation and plugging in the bound in Lemma~\ref{lem:N_n} gives us,
\begin{align*}
\E_{\Omega^\star}[\Pi^n(B_n)] \le e^{(1+\delta/2)(p+|G^\star|)\log n-AM^2n\eps^2_n}\frac{1}{\pi(G^\star)}\sum_G \pi(G)e^{\delta(p+|G|)\log\xi_n} + o(1),
\end{align*}
where the ``$o(1)$'' is by Lemma~\ref{lem:D_n}.  Next, plug in $\pi(G)$ of the form \eqref{eq:prior_G2} to get
\[ \E_{\Omega^\star}[\Pi^n(B_n)] \le e^{(1+\delta/2)(p+|G^\star|)\log n-AM^2n\eps^2_n(G^\star)}e^{\tau_n |G^\star|+\delta p\log\xi_n} \sum_{s=0}^{\bar{R}_n}\binom{\bar{R}_n}{s} e^{-b_n s} + o(1), \]
where $b_n=\tau_n-\delta\log\xi_n$ and $\bar R_n = \binom{p}{2}$.  With $\tau_n=a\log p$ and $\xi_n=p^m$ as in Condition~\ref{cond:P}, we have $b_n=(a-m\delta)\log p$.  By the binomial theorem, 
\[ \sum_{s=0}^{\bar{R}_n}\binom{\bar{R}_n}{s} e^{-b_n s} = (1+e^{-b_n})^{\bar{R}_n} \asymp \exp\{\bar R_n e^{-b_n}\}. \]
Since $a > 1 + m\delta$ and $\bar R_n = O(p^2)$, the right-hand side is of the order 
\[ \exp\{Cp^{2+m\delta-a}\} < \exp\{C'p\}. \]
Putting everything together we get 
\[\E_{\Omega^\star}[\Pi^n(B_n)] \le e^{K(p+|G^\star|)\log p-AM^2n\eps_n^2(G^\star)} + o(1),\]
where $K=K(\delta,m,c,a)$ is a positive constant.  Therefore, if $M$ is sufficiently large, then the upper bound vanishes, proving the first claim of the theorem.  

For the second claim in the theorem, define 
\[ S_n = \{\Omega: \|\Omega - \Omega^\star\|_F > M' \eps_n(G^\star)\}, \]
where $M'$ is allowed to be different from $M$ in the first claim.  By the total probability formula, we can write  
\begin{align}
    \Pi^n(S_n) &= \Pi^n(S_n \cap B_n) + \Pi^n(S_n \cap B_n^c) \le \Pi^n(B_n) + \Pi^n(S_n \cap B_n^c),
    \label{eq:probability_frobenius}
\end{align}
where $B_n^c$ is the complement of $B_n$.  The above proof takes care of $\Pi^n(B_n)$, so here we only need to deal with $\Pi^n(S_n \cap B_n^c)$.  

When $\eps_n(G^{\star})$  vanishes, according to Lemma~A.1 from \citet{banerjee2015bayesian}, $B_n^c$ implies that 
\[\|\Omega^{\star}-\Omega\|_F \le M_{\lambda_0}\eps_n(G^{\star}),\]
where $M_{\lambda_0}$ is a positive constant, depending on the $\lambda_0$ defined in Assumption~\ref{asmp:eigenvalue}. However, $S_n$ requires that
\[\|\Omega^{\star}-\Omega\|_F > M'\eps_n(G^{\star}).\]
Hence, if we let $M'=M_{\lambda_0}$, we have $S_n \cap B_n^c = \varnothing$. In addition, if we take $\eps^2_n(G^{\star})=n^{-1}(p+|G^{\star}|)\log p$, then $\Pi^n(B_n)$ vanishes as $n \to \infty$. This finishes proving the second claim of the theorem.

\ifthenelse{1=1}{}{
\begin{proof}
Since in \eqref{eq:probability_frobenius}, by Theorem~\ref{thm:concentration} we already get a control over $\Pi^n(B_{M\eps_n})$, here we only look at the second term,
\begin{equation}
    \Pi^n(S_{M_{\text{\sc F}}\eps_n} \cap \overline{B_{M\eps_n}})=\frac{N(S_{M_{\text{\sc F}}\eps_n} \cap \overline{B_{M\eps_n}})}{D_n},
    \label{eq:post_probability_frobenius}
\end{equation}
where
\[N(S_{M_{\text{\sc F}}\eps_n} \cap \overline{B_{M\eps_n}})=\sum_{G} \pi(G) \int_{S_{M_{\text{\sc F}}\eps_n} \cap \overline{B_{M\eps_n}} \cap \mathcal{P}_G} R^{\alpha}(\Omega)\pi(\Omega|G)\diff{\Omega},\]
and
 \[D_n=\sum_{G} \pi(G) \int_ {\mathcal{P}_G} R^{\alpha}(\Omega)\pi(\Omega|G)\diff{\Omega}.\]
By Lemma~\ref{lem:D_n}, we have an lower bound for $D_n$. And using the same trick in the proofs for Lemma~\ref{lem:N_n}, we can bound $N(S_{M_{\text{\sc F}}\eps_n} \cap \overline{B_{M\eps_n}})$ above by,
\[\sum_{G} \pi(G) \int_{S_{M_{\text{\sc F}}\eps_n} \cap \overline{B_{M\eps_n}} \cap \mathcal{P}_G} J_n^{\frac{1}{q}}(\Omega)K_n^{\frac{1}{h}}(\Omega)\diff{\Omega},\]
where $J_n(\Omega)$ and $K_n(\Omega)$ are defined in \eqref{eq:J_n} and \eqref{eq:K_n}, $h$ and $q$ are positive constants greater than $1$.
Note that $J_n$ can be rewritten as,
\[J_n=|q\alpha A^{1-q\alpha}+(1-q\alpha) A^{-q\alpha}|^{-\frac{n}{2}},\]
where $A=\Omega^{\star}^{-1/2}\Omega\Omega^{\star}^{-1/2}$. Let the eigenvalues of $A$ be $d_1,...,d_p$, and let $w=q\alpha$, we have
\[j(d_1,...,d_p):=\log J_n(\Omega)=-\frac{n}{2}\sum_{i=1}^p f(d_i),\]
where 
\[f(d_i)=-w\log d_i + \log(1+w(d_i-1)), \quad i=1,..,p.\]
It is easy to show that, for $i=1,...,p$, if $0<w<1$,  first derivative $f'(d_i)<0$ when $d_i \in (0, 1)$; $f'(d_i)>0$ when $d_i \in (1, +\infty)$; and $f'(1)=0$, $f(1)=0$. 
On the other hand, we know that 
\[J_n(\Omega)=\big[1-H^2(p_{\Omega^{\star}},p_{\Omega})\big]^n.\]
Then for $\Omega \in \overline{B_{M\eps_n}}$, we can have 
\[\sum_{i=1}^p f(d_i) \le 2M^2\eps_n^2.\]
Therefore, when $\eps_n \to 0$, we will have $d_i \to 1$.
Hence,  by Taylor series expansion, as $n \to \infty$, for any constant $c_0>0$ there always exists a constant $\delta_0>0$, for $d_1,...,d_i$ satisfying $\underset{i}{\text{max}}|d_i-1|<\delta_0$, \[j(d_1,...,d_p) \le -c_0\frac{nw(1-w)}{4}\sum_{i=1}^p(d_i-1)^2.\]
For $\Omega \in S_{M_{\sc F}\eps_n}$, we have
\begin{align*}
    M_{\sc F}^2\eps_n^2 < \|\Omega-\Omega^{\star}\|_F^2&=\|\Omega^{\star}^{1/2}(I-A)\Omega^{\star}^{1/2}\|_F^2\\
    & \le \lambda^{\star}^2_{\text{max}}\|I-A\|^2_F\\
    &=\lambda^{\star}^2_{\text{max}} \sum_{i=1}^p(d_i-1)^2.
\end{align*}
Therefore, there exists a positive constant $M'$ such that 
\[J_n(\Omega) < e^{-M_{\text{\sc F}}'^2n\eps_n^2}, \quad \text{for} \quad \Omega \in S_{M_{\text{\sc F}
}\eps_n} \cap \overline{B_{M\eps_n}}.\]
Then following the rest of the proof for Lemma~\ref{lem:N_n} and the proof for Theorem~\ref{thm:concentration}, we end up with 
\[\E_{\Omega^{\star}}[\Pi^n(S_{M_{\text{\sc F}}\eps_n})]  \to 0,\]
where $\eps^2_n=n^{-1}(p+|G^{\star}|)\log p \to 0$, as $n \to \infty$.
\end{proof}
}


\subsection{Proof of Theorem~\ref{thm:dimension}}
\label{SS:proof.dimension}

Consider the event $U_n = \{\Omega: |G_\Omega| \geq \Delta\}$.  Like in the proof of Theorem~\ref{thm:concentration}, we can write $\Pi^n(U_n)$ as the ratio $N_n/D_n$, where $D_n$ is as in \eqref{eq:post_probability} but here 
\[ N_n =\sum_{G} \pi(G) \int_{U_n \cap \Pp_G} \E_{\Omega^\star}\bigl\{ R_n^{\alpha}(\Omega) \pi_n(\Omega \mid G) \bigr\} \, \diff{\Omega}.\]
Following the proof of Lemma~\ref{lem:N_n}, it is not difficult to show that 
\begin{equation}
\label{eq:E_U}
\E_{\Omega^\star} (N_n) \leq e^{\delta p \log\xi_n } \sum_{s=\Delta}^{\bar{R}_n} \binom{\bar{R}_n}{s}e^{-b_n s}, 
\end{equation}
where $b_n=\tau_n-\delta\log\xi_n$. Again, for $\tau_n=a\log p$ and $ \xi_n=p^m$, we let
\[\theta_n=\frac{e^{-b_n}}{1+e^{-b_n}}=\frac{1}{1+p^{a-m\delta}}.\]
Next, after some algebraic manipulations to the right-hand side of \eqref{eq:E_U} we find that 
\begin{equation}
\label{eq:tail}
\E_{\Omega^\star} (N_n) \leq (1+p^{m\delta-a})^{\bar{R}_n} \sum_{s=\Delta}^{\bar{R}_n} \binom{\bar{R}_n}{s} \theta_n^s (1-\theta_n)^{\bar{R}_n-s}.
\end{equation}
which involves the tail probability of a binomial distribution, with parameters $\bar R_n$ and $\theta_n$.  Since $\Delta \geq p$ and, by Condition~\ref{cond:P}, the mean $\bar R_n \theta_n$ of the binomial is less than $p$, we can apply a known tail probability bound, e.g., \cite{bahadur1960some} and \cite{klar2000bounds}, to get 
\[ \sum_{s=\Delta}^{\bar{R}_n} \binom{\bar{R}_n}{s} \theta_n^s (1-\theta_n)^{\bar{R}_n-s} \leq \frac{(\Delta + 1)(1-\theta_n)}{(\Delta + 1) - (\bar R_n + 1) \theta_n} f(\Delta), \]
where $f(\Delta)$ is the corresponding binomial mass function evaluated at $\Delta$.  The ratio in the upper bound is clearly $O(1)$, so it suffices to examine $f(\Delta)$.  It is easy to verify that 
\[ f(\Delta) = \frac{(\bar{R}_n\theta_n)^\Delta \prod_{i=0}^{\Delta-1} (1-i/\bar{R}_n)}{\Delta!} e^{-\theta_n(\bar{R}_n-\Delta)} \lesssim 
(\bar{R}_n\pi_n/\Delta)^{\Delta}e^{\Delta-\bar{R}_n\pi_n}. \]
Plugging this bound back into \eqref{eq:E_U}, using the details in Condition~\ref{cond:P}, we get 
\[ \E_{\Omega^\star} (N_n) \leq e^{m\delta p\log p+ p^{2+m\delta-a} + (2+m\delta-a)\Delta \log p-\Delta\log\Delta}. \]
Using Lemma~\ref{lem:D_n}, we can now get an upper bound on $\E_{\Omega^\star} \{\Pi^n(U_n)\}$, i.e., 
\[ \frac{\exp\{m\delta p\log p+ p^{2+m\delta-a} + (2+m\delta-a)\Delta \log p-\Delta\log\Delta\}}{\pi(G^\star)\exp\{-(1+\delta/2)(p+|G^\star|)\log n\}}. \]
It is now easy to verify that, with $\Delta = \rho \max(p, |G^\star|)$, with $\rho$ as in the statement of the theorem, this upper bound vanishes as $n \to \infty$. 



\subsection{Proof of Theorem~\ref{thm:laplace_error}}
\label{SS:laplace_error}

Before examining the accuracy of Laplace approximation, we first prove that posterior normalizing constant $I_G(\delta+\alpha n,\tildeSigma_G)$ can be written in the form of \eqref{eq:I_posterior}. And this suffices to show that 
\[\tr(\hatSigma\Omega)=\tr(\hatOmega_G^{-1}\Omega), \quad \Omega \in \Pp_G.\]
Note that given $G=(V, E)$, $\hatOmega_G$  can be viewed as the solution of the following optimization problem,
\[\underset{\Omega \in \Pp,  \{\upsilon_{ij}: (i, j) \notin E\}}{\max} \quad \log|\Omega|-\tr(\hatSigma\Omega)-\sum_{(i, j) \notin E}\upsilon_{ij}\omega_{ij}.\]
Thus,  \[\hatOmega^{-1}_G=\hatSigma+\Upsilon,\]
where $\Upsilon$ is a matrix of Lagrange parameters with nonzero values for all pairs with edges absent. Then for any $\Omega \in \Pp_G$, we can have $\tr(\Upsilon\Omega)=0$.
Therefore, we have
\[I^{\text{\sc exact}}:=I_G(\delta+\alpha n,\tildeSigma_G)=\int_{\Pp_G} e^{(\delta+\alpha n-2)h(\Omega)/2} \, \diff \Omega.\]

{\color{black}
Laplace approximation for $I^{\text{\sc exact}}$  is based on the Taylor series expansion of $h(\Omega)$ in \eqref{eq:h_fun} around  $\hat{\Omega}_G$ given $G$ such that $|G|=s$. For notational simplicity, we use $\hat{\Omega}:=\hat{\Omega}_G$ as a shorthand for the following statements of Lemma~\ref{lem:taylor_error} and the proofs of Theorem~\ref{thm:laplace_error} and Lemma~\ref{lem:taylor_error}. Denote $D=\Omega-\hat{\Omega}$, where $\Omega, \hat{\Omega} \in \Pp_G$. Let $\Delta=\ve_{G}(D)$, which is the vectorized version of $D_G$, but excluding entries corresponding to the missing edges in the underlying graphical model. Dimension of the vector $\Delta$ here is $p+s$. The following lemma gives a bound for the remainder term in Taylor expansion of $h(\Omega)$.

\begin{lem}
\label{lem:taylor_error}
For any graph G, with probability tending to one, the remainder in the Taylor series expansion of $h(\Omega)$ defined in \eqref{eq:h_fun} can be upper and lower bounded by,
\[\pm \frac{1}{2}(c_1\xi_n^5\|\Delta\|_2^3+\xi_n^4c_2\|\Delta\|_2^4),\]
where $c_1$ and $c_2$ are positive constants, and $\xi_n$ is defined in Condition~\ref{cond:P}. 
\end{lem}
\begin{proof}
See Appendix~\ref{ss:taylor_error}.
\end{proof}


According to Lemma \ref{lem:taylor_error}, we can find an upper bound for the  $h(\Omega)$, i.e.,
\begin{align}
    h(\Omega) \le h(\hat{\Omega})-\frac{1}{2(\delta+\alpha n -2)}\Delta^T H_{\hat{\Omega}} \Delta + \frac{1}{2}(c_1\xi_n^5\|\Delta\|_2^3+c_2\xi_n^4\|\Delta\|_2^4),
    \label{eq:taylor_bound}
\end{align}
where $H_{\hat{\Omega}}=(\delta+\alpha n -2)Q(\hatOmega)$ and $Q(\hatOmega)$ is the negative Hessian matrix of $h(\Omega)$.

Next, we choose a neighborhood of  $\hat{\Omega}$, which is,
\[\mathcal{N}=\{\Delta \in \mathbb{R}^s: \|H_{\hat{\Omega}}^{\frac{1}{2}}\Delta\|_2 \le \zeta_n\},\]
where $\zeta_n$ is a non-decreasing sequence.  Then $I^{\text{\sc exact}}$ can be written as the sum of $I_1^{\text{\sc exact}}+I_2^{\text{\sc exact}}$, where
\begin{align*}
I_1^{\text{\sc exact}} & =\int_{\mathcal{N} \cap \Pp_G} e^{(\delta+\alpha n -2)h(\Omega)/2}\diff\Omega \\
I_2^{\text{\sc exact}} & =\int_{\Pp_G/\mathcal{N}} e^{(\delta+\alpha n-2)h(\Omega)/2}\diff\Omega.
\end{align*}
In order to derive bounds for $I^{\text{\sc exact}}$, we choose to bound $I_1^{\text{\sc exact}}$ and $I_2^{\text{\sc exact}}$ respectively.

First, for  $I_1^{\text{\sc exact}}$, $\Delta \in \mathcal{N}$, given that eigenvalues of the sieve MLE is upper and lower bounded by $\xi_n$ and $\xi_n^{-1}$ respectively,  we have,
\[\|\Delta\|_2 \le \zeta_n\xi_n/\sqrt{(\delta+\alpha n-2)}.\]
Thus,  \eqref{eq:taylor_bound} can be further bounded by,
\[h(\hat{\Omega})-\frac{1}{2(\delta+\alpha n -2)}\Delta^T H_{\hat{\Omega}} \Delta + \frac{1}{2}\xi_n^8\Big[c_1\frac{\zeta_n^3}{(\delta+\alpha n -2)^{\frac{3}{2}}}+c_2\frac{\zeta^4_n}{(\delta+\alpha n -2)^2}\Big].\]
Therefore, $I^{\text{\sc exact}}_1$ can be upper bounded by 
\[ e^{(\delta+\alpha n -2)h(\hat{\Omega})/2}\exp\Big\{\frac{1}{2}\xi_n^8\big[c_1\tfrac{\zeta_n^3}{(\delta+\alpha n-2)^{\frac{1}{2}}}+c_2\tfrac{\zeta^4_n}{\delta+\alpha n-2}\big]\Big\}\int_{\mathcal{N}}\exp\big\{-\frac{1}{4}\Delta^TH_{\hat{\Omega}}\Delta\big\} \diff \Delta.\]
Let $\eta=H_{\hat{\Omega}}^{\frac{1}{2}}\Delta/\sqrt{2}$, then the integral in the right-hand side of the above inequality can be rewritten as,
\begin{align*}
   |H_{\hat{\Omega}}|^{-\frac{1}{2}}& \int_{\|\eta\|_2 < \zeta_n} \exp\{-\frac{1}{2}\|\eta\|_2^2\} \diff \eta =\Big(\frac{(4\pi)^{p+s}}{|H_{\hat{\Omega}}|}\Big)^{\frac{1}{2}}\Pr\{\chi^2_{p+s} < \zeta_n^2\}
\end{align*}
If we let $\zeta_n=\sqrt{5(p+s)\log n }$, according to the tail bound of $\chi^2$ distribution, e.g., in \cite{laurent2000adaptive}, we have
\[\Pr\{\chi^2_{p+s} < \zeta_n^2\} \ge 1-\frac{1}{n^{p+s}}.\]
Finally, we get that $I^{\text{\sc exact}}_1$ can be upper bounded by,
\begin{align}
    e^{(\delta+\alpha n -2)h(\hat{\Omega})/2}\Big(\frac{4\pi}{\delta+\alpha n -2}\Big)^{\frac{p+s}{2}}|Q(\hatOmega)|^{-\frac{1}{2}}\cdot e^\kappa_n,
\end{align}
where $C_1$ and $C_2$ are positive constants and 
\[\kappa_n = \frac{1}{2}\xi_n^8\Big[C_1\sqrt{\frac{(p+s)^3\log^3n}{\delta+\alpha n-2}}+C_2\frac{(p+s)^2\log^2n}{\delta+\alpha n-2}\Big] \]
A lower bound for $I^{\text{\sc exact}}_1$ can obtained similarly.

Next we look look at $I_2^{\text{\sc exact}}$. For $\Delta \notin \mathcal{N}$, we have $\|H_{\hat{\Omega}}^{\frac{1}{2}}\Delta\|_2 \ge \zeta_n$.  Consider a matrix $\tilde{\Omega}$ on the boundary of $\mathcal{N}.$ Then for its corresponding $\tilde{\Delta}=\ve_G(\tilde{\Omega}-\hat{\Omega})$ we have,
\[\|H_{\hat{\Omega}}^{\frac{1}{2}}\tilde{\Delta}\|_2 = \zeta_n.\]
Thus from \eqref{eq:taylor_bound}, we can get,
\begin{align}
    h(\tilde{\Omega})-h(\hat{\Omega}) \le -\frac{1}{2(\delta+\alpha n-2)}\zeta^2_n + \frac{1}{2}\xi_n^8\Big[c_1\frac{\zeta_n^3}{(\delta+\alpha n-2)^{\frac{3}{2}}}+c_2\frac{\zeta^4_n}{(\delta+\alpha n-2)^2}\Big]
    \label{eq:piece1}
\end{align}
Now consider any $\Omega$ with  $\Delta=\ve_G(\Omega-\hat{\Omega})$ satisfying $\|H_{\hat{\Omega}}^{\frac{1}{2}}\Delta\|_2 > \zeta_n$.
And let 
\[\Omega'=\hat{\Omega} + \frac{\zeta_n}{\|H_{\hat{\Omega}}^{\frac{1}{2}}\Delta\|_2} (\Omega-\hat{\Omega}).\]
Vectorizing both sides of the above equation in the same way as $\Delta$ and $\tilde{\Delta}$, and denote $\Delta'=\ve_G(\Omega'-\hat{\Omega})$, we can have,
\[\|H_{\hat{\Omega}}^{\frac{1}{2}}\Delta'\|_2=\zeta_n,\]
i.e., $\Omega'$ is on the surface of the ball.  In addition, because $h(\Omega)$ is concave over $\Pp_G$, it is not hard to show,
\begin{align}
    h(\Omega') \ge h(\hat{\Omega}) + \frac{\zeta_n}{\|H_{\hat{\Omega}}^{\frac{1}{2}}\Delta\|_2} \big[h(\Omega)-h(\hat{\Omega})\big]. 
    \label{eq:piece2}
\end{align}
Then combining \eqref{eq:piece1} and \eqref{eq:piece2} gives us,  
\begin{align}
    h(\Omega) \le h(\hat{\Omega})-\tfrac{1}{2(\delta+\alpha n -2)}\big\{\zeta^2_n -\xi_n^8 \big[c_1\tfrac{\zeta_n^3}{(\delta+\alpha n-2)^{\frac{1}{2}}}+c_2\tfrac{\zeta^4_n}{(\delta+\alpha n-2)}\big]\big\}\|H_{\hat{\Omega}}^{\frac{1}{2}}\Delta\|_2, \quad \forall \Omega \notin \mathcal{N}
\end{align}
Hence, if we let $\eta=H_{\hat{\Omega}}^{\frac{1}{2}}\Delta/\sqrt{2}$,
\[I_2^{\text{\sc exact}} \le e^{(\delta+\alpha n-2)h(\hatOmega_G)/2}|H_{\hat{\Omega}}|^{-\frac{1}{2}}2^{\frac{p+s}{2}}\int_{\|\eta\|_2 > \zeta_n} \exp\big\{-\frac{1}{2}\zeta_n(\zeta_n - w_n)\|\eta\|_2\big\}\diff \eta,\]
where
\[w_n=\xi_n^8\Big[c_1\frac{\zeta_n^2}{(\delta+\alpha n-2)^{\frac{1}{2}}}+c_2\frac{\zeta^3_n}{(\delta+\alpha n-2)}\Big].\]
If we let $\zeta_n=\sqrt{5(p+s)\log n}$, under Assumption~\ref{asmp:dimension} we know that $w_n/\zeta_n \to 0$ as $n \to \infty$, then
by Lemma A.2 from \cite{barber2016laplace},
the integral in the right-hand side of the inequality above can be loosely bounded by,
\[\frac{(2\pi)^{(p+s)/2}(\zeta_n/2)^{p+s-1}}{\Gamma\big((p+s)/2\big)w_n} e^{-\zeta_n^2 +\zeta_n w_n}.\]
Since $\zeta_n=\sqrt{5(p+s)\log n}$, then
\begin{align*}
\zeta_n^2 - & \Big[c_1\frac{\zeta_n^3\xi_n^8}{(\delta+\alpha n-2)^{\frac{1}{2}}}+c_2\frac{\zeta^4_n\xi_n^8}{\delta+\alpha n-2}\Big] \\
& \simeq (p+s)\log n\Big[1-\Big(c_1\sqrt{\frac{(p+s)\xi_n^6\log n}{\delta+\alpha n-2}}+c_2\frac{(p+s)\xi_n^8\log n }{\delta+\alpha n-2}\Big)\Big].
\end{align*}
A corresponding lower bound for $I^{\text{\sc exact}}_2$ can be obtained similarly.

Putting everything together, we find that the posterior normalizing constant $I^{\text{\sc exact}}$ can be upper and lower bounded by,
\[e^{(\delta+\alpha n -2)h(\hat{\Omega})/2}\big(\frac{4\pi}{\delta+\alpha n -2}\big)^{\frac{p+s}{2}}|Q(\hatOmega)|^{-\frac{1}{2}}\cdot e^{\pm\kappa_n},\]
where 
\[\kappa_n = \frac{1}{2}\xi_n^8\Big[C_1\sqrt{\frac{(p+s)^3\log^3n}{\delta+\alpha n-2}}+C_2\frac{(p+s)^2\log^2n}{\delta+\alpha n-2}\Big]\Big.\]
Under Assumption~\ref{asmp:dimension}, the approximation error, which is controlled by $\kappa_n$, vanishes if $n^{-1/2}\xi_n^8(p+|G|)^{3/2}\log^{3/2} p \to 0$. 


}

\section{Other proofs}

\subsection{Proof of Lemma~\ref{lem:D_n}}
\label{proof:D_n}

Since $D_n$ in \eqref{eq:post_probability} is a sum of nonnegative terms, it is clearly greater than the single term when $G=G^{\star}$, i.e.,
\begin{align}
        D_n > \pi(G^\star)\int_{\Pp_G} R_n^\alpha(\Omega) \pi_n(\Omega \mid G) \, \diff \Omega 
      = \pi(G^\star) L_n^{-\alpha}(\Omega^\star)\frac{I_{G^\star}(\delta+\alpha n, \tilde{\Sigma}_{G^\star})}{I_{G^\star}(\delta, (\delta-2) \hat{\Omega}_{G^\star}^{-1})}.
      \label{eq:D_n}
\end{align}
Let $\mathcal{M}_G$ denote the set of all $|V|\times|V|$ matrices $A=(A_{ij})_{1\le i,j \le |V|}$ satisfying $A_{ij}=A_{ji}=0$ for all pairs $(i,j) \notin E$ and $i \neq j$. Clearly, $\Pp_G \subset \mathcal{M}_G$.  Recall that an undirected graph $G = (V,E)$ is decomposable if and only if there exists a permutation of vertices V such that after reordering the vertices based on this permutation, every $\Omega \in \Pp_G$ factors as $\Omega=LL^\top$ where $L \in \mathcal{M}_G$ and $L$ is lower triangular with positive diagonal entries. Such permutation is called a perfect vertex elimination scheme for $G$. See \cite{roverato2000} for more details. This property basically says that for a decomposable graph $G$, if the vertices are ordered according to a  perfect vertex elimination scheme and $\Omega=LL^\top$, $\Omega \in \Pp_G$, then $L$ has the same zero pattern as $\Omega$ in its lower triangle. Here, without loss of generality, we always assume that the vertices of the graphs discussed below have already been reordered by the perfect vertex elimination scheme.

Next, to find an lower bound for $D_n$, we look at the ratio $\mathcal{R}_n$ of normalizing constants in \eqref{eq:D_n}, which given by 
\begin{align*}
\mathcal{R}_n & = \frac{I_{G^\star}(\delta+\alpha n, \tildeSigma_{G^\star})}{I_{G^\star}(\delta, (\delta-2) \hatOmega_{G^\star}^{-1})} \\
& = \frac{\int_{\Pp_{G^\star}}|\Omega|^{\frac{\delta+ \alpha n-2}{2}}\exp\big\{-\frac{1}{2}\tr[(\alpha n\hatSigma+(\delta-2)\hatOmega_{G^\star}^{-1})\Omega]\big\} \, \diff\Omega}{\int_{\Pp_{G^\star}}|\Omega|^{\frac{\delta-2}{2}}\exp\big\{-\frac{(\delta-2)}{2}\tr(\hatOmega_{G^{\star}}^{-1}\Omega)\big\} \, \diff\Omega}.
\end{align*}
It is easy to show that, for $\hatSigma=n^{-1} X^\top X$ and $\hatOmega_G$ in \eqref{eq:likelihood}, we have
\[ \tr(\hat{\Sigma}\Omega)=\tr(\hat{\Omega}_G^{-1}\Omega), \quad \text{for all $\Omega \in \Pp_G$}. \]
Next, make a change of variables, namely, 
\[ U=\{\alpha+n^{-1} (\delta-2)\} \, \Omega \quad \text{and} \quad V=n^{-1}(\delta-2) \, \Omega, \]
so that 
\[ \diff U=\{\alpha+n^{-1}(\delta-2)\}^{-(p+|G^\star|)} \, \diff \Omega \quad \text{and} \quad \diff V=\{n^{-1}(\delta - 2)\}^{-(p+|G^\star|)} \, \diff \Omega. \]
Then the ratio $\mathcal{R}_n$ can be rewritten as 
\[ \mathcal{R}_n = A_n \frac{\int_{\Pp_{G^\star}}|U|^{\frac{\delta+\alpha n -2}{2}}\exp\big\{-\frac{n}{2}\tr(\hatOmega_{G^\star}^{-1}U)\big\} \diff U}{\int_{\Pp_{G^\star}}|V|^{\frac{\delta-2}{2}}\exp\big\{-\frac{n}{2}\tr(\hatOmega_{G^\star}^{-1}V)\big\}\diff V} = A_n \E\big[|U|^{n\alpha/2}\big], \]
where $U \sim \wish_{G^\star}(\delta, n\hatSigma_{G^\star})$, with $\hatSigma_{G^\star} = \hatOmega_{G^\star}^{-1}$, and
\[ A_n = [\alpha+(\delta-2)/n]^{-\frac{n\alpha p}{2}}[1+\alpha n/(\delta-2)]^{-(\delta-2) p/2-(p+|G^\star|)}. \]
The right-most expression for $\mathcal{R}_n$ involves a certain moment of the $G$-Wishart distribution on the graph $G^\star$.  Since we are assuming that $G^\star$ is decomposable, an expression this moment can be derived by using the following distributional result from \citet{xiang2015high}, Lemma~3.3:
\begin{quote}
If $G=(V,E)$ is a decomposable graph, $\Omega \sim \wish_G(\delta, D)$ and $\Omega=LL^\top$, then $L_{ji}=0$ for $1 \le i <j \le p$ and $(i,j) \notin E$, $L_{.i}^{\ge}=(L_{ij})_{j\ge i, (i,j) \in E}$ are independent for $1 \le i \le p$, and
\begin{align*}
L^2_{ii} & \sim \gam\big((v_i+\delta)/2+1,nc_i/2\big)\\
L^{>}_{.i} \mid L_{ii} & \sim \nm\big(-(D^{>i})^{-1}D^{>}_{.i}L_{ii}, \, n^{-1}(D^{>i})^{-1}\big),
\end{align*}
where $v_i=\text{dim}(L^{>}_{.i})$, $c_i=n^{-1} \big(D_{ii}-(D^{>}_{.i})^T(D^{>i})^{-1}D^{>}_{.i}\big)$, $D^{>}_{.i}=(D_{ji})_{j>i, (i,j) \in E}$, and $D^{>i}= (D_{jk})_{i<j, k\le p, (i,j) \in E, (i,k) \in E}$.
\end{quote}
From this, we obtain the expression 
\begin{align*}
 \E\big[|U|^{n\alpha/2}\big]&=\prod\limits_{i=1}^p\E\big[(L^2_{ii})^{n\alpha/2}\big]\\
&=\prod_{i=1}^p (nc_i/2)^{-n\alpha/2}\frac{\Gamma\big((n\alpha+(\delta-2)+v_i)/2+1\big)}{\Gamma\big(((\delta-2)+v_i)/2+1\big)}.
\end{align*}
By Stirling's formula,
\begin{align*}
    \Gamma\big((n\alpha+(\delta-2)+v_i)/2+1\big)
    & = \Big\{\frac{2\pi}{(n\alpha+(\delta-2)+v_i)/2+1}\Big\}^{1/2} \\
    &\quad \times \Big\{\frac{(n\alpha+(\delta-2)+v_i)/2+1}{e}\Big\}^{(n\alpha+(\delta-2)+v_i)/2+1} \\
    &\quad \times \{1+O(n^{-1})\}
\end{align*}
Since $p = o(n)$, the $O(n^{-1})$ errors in the individual Stirling approximations do not accumulate when taking the product of $p$ terms, so we get 
\begin{align*}
    \E\big[|U|^{n\alpha/2}\big] \approx e^{-n\alpha p/2} & \prod_{i=1}^p \Big\{ \frac{\alpha+(\delta-2)/n}{c_i}+\frac{v_i+2}{nc_i}\Big\}^{n\alpha/2} \\
& \times 
   \prod_{i=1}^n \frac{[\{n\alpha+(\delta-2)+v_i\}/2+1]^{\{(\delta-2)+v_i\}/2+1}}{\Gamma(\{(\delta-2)+v_i\}/2+1)} \\
    & \times \prod_{i=1}^n e^{-((\delta-2)+v_i)/2-1} \Big[\frac{2\pi}{\{n\alpha+(\delta-2)+v_i\}/2+1} \Big]^{1/2}.
\end{align*}
Using the recursive property of gamma function, it is easy to show that,
\[ \frac{[\{n\alpha+(\delta-2)+v_i\}/2+1]^{\{(\delta-2)+v_i\}/2+1}}{\Gamma(\{(\delta-2)+v_i\}/2+1)} \geq 1 \quad \text{as $n \to \infty$}. \]
Therefore, putting everything back together gives us 
\begin{align} 
\mathcal{R}_n 
& \ge e^{-n\alpha p/2} [1+\alpha n/(\delta-2)]^{-(\delta-2) p/2-(p+s^\star)} \prod_{i=1}^p \Bigl\{ \frac{1}{c_{i}}+\frac{v_i+2}{c_i(\alpha n+\delta-2)} \Bigr\}^{n\alpha/2} \notag \\
& \quad \times \prod_{i=1}^n  e^{-((\delta-2)+v_i)/2-1} \Bigl\{ \frac{2\pi}{(n\alpha+\delta-2+v_i)/2+1} \Bigr\}^{1/2} \notag\\
&\ge e^{-n\alpha p/2}[1+\alpha n/(\delta-2)]^{-(\delta-2) p/2-(p+s^{\star})} \prod_{i=1}^p (c_i^{-1})^{n\alpha/2} \notag\\
& \quad \times e^{-((\delta-2)+\bar{v})p/2-p}\Bigl\{ \frac{2\pi}{(n\alpha+\delta-2+\bar{v})/2+1} \Bigr\}^{p/2} \label{eq:ratio},
\end{align}
where $\bar{v}=\max\{v_i\}=O(p^{-1}|G^\star|)$, $c_i=\hatSigma_{G^\star ii}-(\hat\Sigma^{>}_{G^\star \dot i})^T(\hatSigma_{G^\star}^{>i})^{-1} \hatSigma^{>}_{G^\star \dot i}$ with $\hatSigma_{G^\star}=\hatOmega^{-1}_{G^\star}$. 

By Lemma 3.2 from \citet{xiang2015high}, consider Choleksy decomposition of $\Omega=LL^\top \in \Pp_G$, and let $\Sigma=\Omega^{-1}$.  Then we have 
\[L_{ii}=\big(\Sigma_{ii}-(\Sigma^{>}_{\dot i})^\top(\Sigma^{>i})^{-1}\Sigma^{>}_{\dot i}\big)^{-1/2},\]
and thus we have $c_i=\hat{L}_{G^\star ii}^{-2}$, where $\hat{L}_{G^\star}$ is the Cholesky factor of $\hatOmega_{G^\star}$.  Given that $\prod_{i=1}^p (\hat{L}_{G^\star ii}^{2})^{n\alpha/2}= |\hatOmega_{G^\star}|^{n\alpha/2}$, we can further bound  \eqref{eq:ratio} from below by,
\[|\hat{\Omega}_{G^{\star}}|^{n\alpha/2}[1+\alpha n/(\delta-2)]^{-(\delta-2) p/2-(p+s^{\star})}e^{-(\delta+\alpha n-2+\bar{v}) p/2-p} \big(\tfrac{2\pi}{(n\alpha+(\delta-2)+\bar{v})/2+1}\big)^{p/2}.\]
Therefore, 
\begin{align*}
D_n &\gtrsim \pi(G^\star)L_n^{-\alpha}(\Omega^\star)|\hatOmega_{G^\star}|^{n\alpha/2}[1+\tfrac{\alpha n}{\delta-2}]^{-(\delta-2) p/2-(p+|G^\star|)}e^{-(\delta+\alpha n-2+\bar{v}) p/2}e^{-p\log n}\\
    &=\pi(G^\star)\frac{|\hatOmega_{G^\star}|^{n\alpha/2}\exp\{-\frac{n\alpha p}{2}\}}{|\Omega^\star|^{n\alpha/2}\exp\{-\frac{n\alpha}{2}\tr(\hatOmega_{G^\star}^{-1}\Omega^\star)\}} [1+\tfrac{\alpha n}{\delta-2}]^{-(\delta-2) p/2-(p+|G^\star|)}e^{-((\delta-2)+\bar{v})p/2}e^{-p\log n}\\
    &\gtrsim \pi(G^\star)\exp\big\{-(1+\delta/2)(p+|G^\star|)\log n\big\}
\end{align*}
as was to be shown.

\subsection{Proof of Lemma~\ref{lem:N_n}}
\label{proof:N_n}

Apply H{\"o}lder's inequality to the inter expectation in right-hand side of \eqref{eq:EN_n}, i.e., for constants $h>1$ and $q=\frac{h}{h-1}$ chosen so that $\alpha q \in [\frac12, 1)$, we aim to get an upper bound for 
\begin{equation}
\label{eq:N.bound}
\E_{\Omega^{\star}}(N_n) \le \sum_{G} \pi(G) \int_{B_n \cap \Pp_G} J_n^{1/q}(\Omega) \, K_n^{1/h}(\Omega) \, \diff{\Omega},
\end{equation}
where 
\[ J_n(\Omega) = \E_{\Omega^\star} \{R_n^{\alpha q}(\Omega)\} \quad \text{and} \quad K_n(\Omega) = \E_{\Omega^\star}\{ \pi_n^h(\Omega \mid G)\}. \]
It is easy to see that $J_n(\Omega)$ is related to the R\'enyi $\alpha q$-divergence of $p_\Omega^n$ from $p_{\Omega^\star}^n$ and, by Theorem~16 in \citet{van2014renyi}, this can be related to the Hellinger distance.  Indeed, if $\Omega \in B_n$, so that $H(p_{\Omega^\star}, p_\Omega) > 1-e^{M^2 \eps_n^2(G^\star)}$, then 
\[ J_n(\Omega) < e^{-2(1-\alpha q) M^2 n \eps_n^2(G^\star)}. \]
Next, for $K_n$, the challenge is in evaluating an expectation with respect to the distribution of $\hatOmega_G$.  Here we make use of the proposed sieve's structure to get a data-independent upper bound for the empirical prior density $\pi_n(\Omega \mid G)$.  That is, for $\hatOmega_g \in \Theta_n(G)$ as in \eqref{eq:sieve}, it can be shown (see Appendix~\ref{SS:trace}) that
\begin{equation}
\label{eq:trace}
\xi_n^{-1}\tr(\Omega) \le \tr(\hatOmega_{G}^{-1}\Omega) \le \xi_n\tr(\Omega).
\end{equation}
Then we can upper bound the integral of $K_n^h$ by,
\[ \frac{\int_{\Pp_G}|\Omega|^{(\delta-2)/2}\exp\big\{-\frac{(\delta-2)}{2}\tr(\xi_n^{-1}\Omega)\big\}\, \diff \Omega}{\int_{\Pp_G}|\Omega|^{(\delta-2)/2}\exp\big\{-\frac{(\delta-2)}{2}\tr(\xi_n \Omega)\big\} \, \diff \Omega}.\]
Making a change of variable 
\[ U=\xi_n^{-1}\Omega \quad \text{and} \quad V=\xi_n\Omega, \]
so that 
\[ \diff U=\xi_n^{-(p+|G|)} \, \diff \Omega \quad \text{and} \quad \diff V=\xi_n^{p+|G|} \diff \Omega,\]
the we can further bound the above expression by 
\[\xi_n^{p(\delta-2)+2(p+|G|)}\frac{\int_{\Pp_G}|U|^{(\delta-2)/2}\exp\big\{-\frac{(\delta-2)}{2}\tr(U)\big\}\diff U}{\int_{\Pp_G}|V|^{(\delta-2)/2}\exp\big\{-\frac{(\delta-2)}{2}\tr(V)\big\}\diff V} = \xi_n^{p(\delta-2)+2(p+|G|)}. \]
Plugging everything back into the right-hand side of \eqref{eq:N.bound} gives 
\[ \E_{\Omega^{\star}}(N_n) \le e^{-AMn\epsilon_n^2} \sum_{G} \pi(G) \xi_n^{p(\delta-2)+2(p+|G|)}, \]
where $A=2(1-\alpha q)/q$, as was to be shown.

\subsection{Proof of Lemma~\ref{lem:taylor_error}}
\label{ss:taylor_error}

{\color{black}
First, let $\|\cdot \|_{sp}$ and $\|\cdot \|_F$  represent matrix spectral norm and Frobenius norm respectively.  Let $\vec{\Delta}=\text{vec}(D)$ be the regular version of vectorization of matrix $D$, i.e., by stacking the columns of $D$ on top of one another. Recall that we also define $\Delta=\ve_G(D)$, which is also a vectorized
version of $D$, but excluding entries corresponding to the missing edges in
the underlying graphical model $G$. Note that $\vec{\Delta}$ is a $p^2$-dimensional vector with $p+2s$ nonzero entries and $\Delta$ is a vector of $p+s$ length, where $s=|G|$.

For $t \in (0,1)$, let us consider the Taylor series expansion of $h(\Omega)$, using Lagrange's form of the
remainder,

\begin{align}
    & h(\Omega)=h(\hat{\Omega})-\frac{1}{2}\vec{\Delta}^T [\hat{\Omega}^{-1} \otimes \hat{\Omega}^{-1}] \vec{\Delta} + R_n \label{eq:taylor1} \\
    & h(\Omega)=h(\hat{\Omega})-\frac{1}{2}\vec{\Delta}^T [(\hat{\Omega}+tD)^{-1} \otimes (\hat{\Omega}+tD)^{-1}] \vec{\Delta} \label{eq:taylor2}
\end{align}
Thus we have,
\begin{align*}
    |R_n| \le \frac{1}{2}\|\vec{\Delta}\|_2^2 \|(\hat{\Omega}+tD)^{-1} \otimes (\hat{\Omega}+tD)^{-1}-\hat{\Omega}^{-1} \otimes \hat{\Omega}^{-1}\|_{sp}.
\end{align*}

Now, let's look at $(\hat{\Omega}+tD)^{-1} \otimes (\hat{\Omega}+tD)^{-1}-\hat{\Omega}^{-1} \otimes \hat{\Omega}^{-1}$, using some linear algebra of Kronecker product, we can get
\begin{align*}
    (\hat{\Omega}+tD)^{-1} \otimes (\hat{\Omega}+tD)^{-1}-\hat{\Omega}^{-1} \otimes \hat{\Omega}^{-1}
    =\big[\big(I+t(\hat{\Omega} \otimes \hat{\Omega})^{-1}U\big)^{-1}-I\big](\hat{\Omega} \otimes \hat{\Omega})^{-1},
\end{align*}
where
\[U=\hat{\Omega}\otimes D+D\otimes\hat{\Omega}+tD\otimes D.\]
Then Woodbury formula gives us
\[(\hat{\Omega}+tD)^{-1} \otimes (\hat{\Omega}+tD)^{-1}-\hat{\Omega}^{-1} \otimes \hat{\Omega}^{-1}=-t\big(I+t(\hat{\Omega} \otimes \hat{\Omega})^{-1}U\big)^{-1}(\hat{\Omega} \otimes \hat{\Omega})^{-1}U(\hat{\Omega} \otimes \hat{\Omega})^{-1}\]
According to submultiplicativity  of operator norm, we have
\begin{align*}
    &\|(\hat{\Omega}+tD)^{-1} \otimes (\hat{\Omega}+tD)^{-1}-\hat{\Omega}^{-1} \otimes \hat{\Omega}^{-1}\|_{sp}\\
    &\le t\|(\hat{\Omega} \otimes \hat{\Omega})^{-1}\|_{sp}^2\|U\|_{sp}\|\big(I+t(\hat{\Omega} \otimes \hat{\Omega})^{-1}U\big)^{-1}\|_{sp}.
\end{align*}

Next, let us examine the upper bound for each term in the right-hand side of the above inequality.

Firstly, by the sieve restriction in \eqref{eq:sieve}, we have 
\[\|(\hat{\Omega} \otimes \hat{\Omega})^{-1}\|_{sp}^2 \le \xi_n^4.\]

Secondly, by triangular inequality,
\begin{align*}
    \|U\|_{sp} &\le \|\hat{\Omega}\otimes D\|_{sp}+\|D\otimes\hat{\Omega}\|_{sp}+\|D\otimes D\|_{sp}.
\end{align*}
Note that, $\|A \otimes B\|_{sp}=\|A\|_{sp}\|B\|_{sp}$ \citep{lancaster1972norms} and $\|A\|_{sp} \le \|A\|_F$,
$\|U\|_{sp}$ can be further upper bounded by,
\[c'_1\xi_n\|D\|_F+c'_2\|D\|_F^2,\]
where $c'_1$ and $c'_2$ are positive constants.

Thirdly, by mixed-product property of Kronecker product operator,
\begin{align*}
    \big(I_{p^2\times p^2}+t(\hat{\Omega} \otimes \hat{\Omega})^{-1}U\big)^{-1}&=(I_p\otimes I_p+tI\otimes (\hat{\Omega}^{-1}D)+t(\hat{\Omega}^{-1}D)\otimes I+t^2(\hat{\Omega}^{-1}D)\otimes(\hat{\Omega}^{-1}D))^{-1}\\
    &=\big((I + t\hat{\Omega}^{-1}D)\otimes(I + t\hat{\Omega}^{-1}D)  \big)^{-1}.
\end{align*}
Thus,
\[\|\big(I+t(\hat{\Omega} \otimes \hat{\Omega})^{-1}U\big)^{-1}\|_{sp}=\|(I + t\hat{\Omega}^{-1}D)^{-1}\|^2_{sp}.\]
Given that $D=\Omega-\hat{\Omega}$, for $t \in (0,1)$, we have
\[\|(I + t\hat{\Omega}^{-1}D)^{-1}\|_{sp} \le \frac{1}{\|(1-t)I+t\hat{\Omega}^{-1}\Omega\|_{sp}}.\]
Since $\hat{\Omega}$ and $\Omega$ are both positive definite matrices, i.e., $\lambda_{\text{min}}(\hat{\Omega}^{-1}\Omega) > 0$, $\|(I + t\hat{\Omega}^{-1}D)^{-1}\|^2_{sp}$ can be further bounded by \[\frac{1}{1-t}.\]

Therefore, we can finally have,
\[|R_n| \le \frac{1}{2}(c''_1\xi_n^5\|D\|_F^3+c''_2\xi_n^4\|D\|_F^4),\]
where $c''_1$ and $c''_2$ are positive constants.

Note that \[\|D\|_F=\|\vec{\Delta}\|_2 \le \sqrt{2}\|\Delta\|_2, \]
therefore, 
\[|R_n| \le \frac{1}{2}(c_1\xi_n^5\|\Delta\|_2^3+c_2\xi_n^4\|\Delta\|_2^4),\]
where $c_1$ and $c_2$ are positive constants.
}

\subsection{Proof of (\ref{eq:trace})}
\label{SS:trace}

Let $\hatOmega$ be such that its eigenvalues are bounded in $[\xi^{-1}, \xi]$, for $\xi > 1$.  We are interested in bounding $\tr(\hat\Omega^{-1}\Omega)$ for a generic $\Omega$.  We can rewrite this as 
\[ \tr(\hatOmega^{-1} \Omega) = \tr(\Omega^{1/2} \hatOmega^{-1} \Omega^{1/2}) = \tr(\Omega^{1/2} \Gamma \Lambda^{-1} \Gamma^\top \Omega^{1/2}), \]
where $\hatOmega = \Gamma \Lambda \Gamma^\top$ is the spectral decomposition, with $\Lambda = \text{diag}(\lambda_1,\ldots,\lambda_p$ the diagonal matrix of the eigenvalues of $\hatOmega$, all of which are bounded in $[\xi^{-1}, \xi]$.  Set $A=\Omega^{1/2} \Gamma$, so that $\tr(\hatOmega^{-1} \Omega) = \tr(A \Lambda A^\top)$.  For $A$ with structure 
\[A=\begin{pmatrix} a_{11} & a_{12} & \dots & a_{1p}\\ a_{21} & a_{22} & \dots &a_{2p}\\\vdots & \vdots &\dots &\vdots \\ a_{p1} & a_{p2} &\dots &a_{pp} \end{pmatrix},\]
we get 
\[\tr(A\Lambda A^T)= \sum_{j=1}^p \lambda_j \sum_{i=1}^p a_{ij}^2 \le \max(\Lambda) \sum_{i=1}^p \sum_{j=1}^p a_{ij}^2 \le \xi \tr(AA^\top),\]
and, similarly, $\tr(A\Lambda A^\top) \ge \xi^{-1}\tr(AA^\top)$.  By observing that $AA^\top = \Omega$, \eqref{eq:trace} holds.

\bibliographystyle{apalike}
\bibliography{ref}

\end{document}